\documentclass[11pt,reqno]{amsart}
\usepackage{graphicx}
\usepackage{pifont}
\usepackage{asymptote}
\usepackage{asypictureB}
\usepackage{enumitem}
\usepackage{comment}
\usepackage{caption}
\usepackage{setspace}
\setlist[enumerate]{font={\upshape}, label=\arabic*., leftmargin=2.5em}
\setlist[itemize]{leftmargin=2.5em}
\setlist[description]{leftmargin=\parindent, 
	itemsep=3pt
}

\newlist{equivlist}{enumerate}{1}
\setlist[equivlist]{font={\upshape}, label=(\roman*)}
\usepackage{tikz}
\usetikzlibrary{matrix,intersections,calc}
\tikzset{ 
	table/.style={
		matrix of nodes,
		nodes={rectangle,text width=1.75em,align=center},
		text depth=1.25ex,
		text height=2.5ex,
		nodes in empty cells
	}
}
\usepackage{mathtools}
\usepackage{amsfonts}
\usepackage{amssymb}
\usepackage{cite} 
\usepackage[sort,compress]{natbib}

\usepackage[pdftex
,pagebackref
]{hyperref}
\hypersetup{
	colorlinks=true,
	allcolors=blue
}
\usepackage[capitalise]{cleveref}
\usepackage[margin=1in]{geometry}
\newtheorem{theorem}{Theorem}[section]
\newtheorem{lemma}[theorem]{Lemma}
\newtheorem{claim}{Claim}[theorem]
\Crefname{claim}{Claim}{Claims}
\newenvironment{clm}[1]
{\renewcommand\theclaim{#1}\claim}
{\endclaim}

\newtheorem{conjecture}[theorem]{Conjecture}
\Crefname{conjecture}{Conjecture}{Conjectures}

\expandafter\let\expandafter\oldproof\csname\string\proof\endcsname
\let\oldendproof\endproof
\renewenvironment{proof}[1][\proofname]{%
	\oldproof[\normalfont\bfseries #1]%
}{\oldendproof}
\newenvironment{subproof}[1][\normalfont\it Subproof]{%
	\begin{proof}[#1]%
	}{%
	\end{proof}%
}

\newcommand{\dd}{\textquotedblleft}
\newcommand{\ee}{\textquotedblright}

\newcommand{\mac}{\mathcal}
\newcommand{\mab}{\mathbb}

\newcommand{\eps}{\varepsilon}

\newcommand{\nin}{\notin}

\renewcommand{\subset}{\subseteq}
\renewcommand{\supset}{\supseteq}

\newcommand{\erh}{Erd\H{o}s--Hajnal}
\newcommand{\gas}{Gy\'{a}rf\'{a}s--Sumner}

\DeclarePairedDelimiter\abs{\lvert}{\rvert}%
\makeatletter
\newcommand{\leqnomode}{\tagsleft@true}
\newcommand{\reqnomode}{\tagsleft@false}
\makeatother

\date{September 17, 2024; revised \today}
\begin{document}
	\title{On polynomially high-chromatic pure pairs}
	\author{Tung H. Nguyen}
	\address{Princeton Univeristy, Princeton, NJ 08544, USA}
	\email{\href{mailto:tunghn@math.princeton.edu}{tunghn@math.princeton.edu}}
	\thanks{Partially supported by AFOSR grant FA9550-22-1-0234, NSF grant DMS-2154169, and
		a Porter Ogden Jacobus Fellowship.}
	\begin{abstract}
		Let $T$ be a forest.
		We study polynomially high-chromatic pure pairs in graphs with no $T$ as an induced subgraph ({\em $T$-free} graphs in other words),
		with applications to the polynomial \gas{} conjecture.
		In addition to reproving several known results in the literature, we deduce: 
		\begin{itemize}
			\item If $T=P_5$ is the five-vertex path, then every $T$-free graph $G$ with clique number $w\ge2$
			contains a complete pair $(A,B)$ of induced subgraphs with $\chi(A)\ge w^{-d}\chi(G)$ and $\chi(B)\ge 2^{-d}\chi(G)$, for some universal $d\ge1$.
			The proof uses the recent \erh{} result for $P_5$-free graphs.
			Via the classical Gy\'arf\'as path argument, such a \dd polynomial versus linear high-$\chi$ complete pairs\ee{} result can be viewed as further supporting evidence for the polynomial \gas{} conjecture for $P_5$.
			In particular, it implies
			\[\chi(G)\le w^{O(\log w/\log\log w)}\]
			which asymptotically improves the bound $\chi(G)\le w^{\log w}$ of Scott, Seymour, and Spirkl.
			
			\item If $T$ and a broom satisfy the polynomial \gas{} conjecture, then so does their disjoint union.
			Unifying earlier results of Chudnovsky, Scott, Seymour, and Spirkl, and of Scott, Seymour, and Spirkl, this gives new instances of $T$ for which the conjecture holds.
			
		\end{itemize} 
	\end{abstract}
	\maketitle
\section{Introduction}
	All graphs in this paper are finite and simple.
	For an integer $k\ge2$, let $P_k$ denote the $k$-vertex path.
	For a graph $G$, let $\abs G$ denote the number of vertices of $G$.
	For every $v\in V(G)$, let $N_G(v)$ be the set of neighbours of $v$ in $G$, and let $N_G[v]:=N_G(v)\cup\{v\}$.
	The {\em chromatic number} of $G$, denoted by $\chi(G)$, is the least $\ell\ge0$ such that the vertex set $V(G)$ of $G$ can be partitioned into $\ell$ stable sets in $G$;
	the {\em clique number} of $G$, denoted by $\omega(G)$, is the size of a largest clique in $G$;
	and $\alpha(G)$ is the size of a largest stable set in $G$.
	For a graph $G$ with $S\subset V(G)$,
	let $G[S]$ be the subgraph of $G$ induced on $S$, and write $\chi(S)$ for $\chi(G[S])$ when there is no danger of ambiguity.
	For disjoint $A,B\subset V(G)$,
	the pair $(A,B)$ is {\em complete} if $G$ contains all possible edges between $A$ and $B$,
	is {\em anticomplete} if $G$ has no edge between $A$ and $B$,
	and {\em pure} if $(A,B)$ is either complete or anticomplete in $G$.
	An {\em induced subgraph} of $G$ is a graph obtained from $G$ by removing vertices;
	and say that $G$ is {\em $H$-free} for some graph $H$ if $G$ has no induced subgraph isomorphic to $H$.
	A class $\mac G$ of graphs is {\em hereditary} if it is closed under isomorphism and taking induced subgraphs.
	We say that $\mac G$ is {\em $\chi$-bounded} if there exists $f\colon\mab N\to\mab N$ depending on $\mac G$ only such that
	$\chi(G)\le f(\omega(G))$ for all $G\in\mac G$;
	and such a function $f$ is {\em $\chi$-binding} for $\mac G$
	(see~\cite{MR4174126,scott2022} for surveys on $\chi$-boundedness).
	The Gy\'arf\'as--Sumner conjecture~\cite{MR382051,MR634555} asserts that:
\begin{conjecture}
	[\gas{}]
	\label{conj:gs}
	For every forest $T$, the class of $T$-free graphs is $\chi$-bounded.
\end{conjecture}
	This conjecture remains largely open and is known to hold for a few restricted families of forests; see~\cite[Section 3]{MR4174126} and the references therein, and~\cite{stable1} for a recently obtained approximation.
	
	We say that a hereditary class $\mac G$ is {\em polynomially $\chi$-bounded} if  $\mac G$ is {$\chi$-bounded} with a polynomial $\chi$-binding function.
	While most known partial results towards \cref{conj:gs} yield super-exponential $\chi$-binding functions and it is known~\cite{MR4707561} that there are $\chi$-bounded classes of graphs that are not polynomially $\chi$-bounded,
	the following substantial strengthening of \cref{conj:gs} could be true:
\begin{conjecture}
	[Polynomial \gas]
	\label{conj:pgs}
	For every forest $T$, the class of $T$-free graphs is polynomially $\chi$-bounded.
\end{conjecture}
(The case when $T$ is a path was independently asked by Esperet~\cite{esperet} and by Trotignon and Pham~\cite{MR3789676}.)
This conjecture is of particular interest because of a conjecture of Erd\H os and Hajnal~\cite{MR1031262,MR599767} that for every (not necessarily forest) graph $H$, every $H$-free graph $G$ satisfies $\max(\alpha(G),\omega(G))\ge \abs {G}^c$ for some $c>0$ depending on $H$ only.
Since $\chi(G)\ge\abs{G}/\alpha(G)$ by definition,
if \cref{conj:pgs} holds for $H$ then so does the \erh{} conjecture.
Currently, the five-vertex path $P_5$ is the smallest open case of \cref{conj:pgs};
and recently, it has been proved that $P_5$ does satisfy the \erh{} conjecture~\cite{density7}:
\begin{theorem}
	[Nguyen--Scott--Seymour]
	\label{thm:ehp5}
	There exists $a\ge4$ such that for every $k\ge1$, every $P_5$-free graph with more than $k^a$ vertices has a clique or stable set with more than $k$ vertices.
\end{theorem}
	The first goal of this paper is to provide an improved bound on the $\chi$-binding functions of $P_5$-free graphs.
	As is well-known, the Gy\'arf\'as path argument~\cite{MR382051} (see \cref{thm:gs}) implies that every $P_5$-free graph with clique number at most $w\ge2$ has chromatic number at most $3^w$;
	and Esperet, Lemoine, Maffray, and Morel~\cite{MR3010736} pushed this down slightly to $5\cdot 3^{w-3}$.
	Scott, Seymour, and Spirkl~\cite{MR4648583} recently improved this exponential bound to $w^{\log w}$.
	(In this paper $\log$ denotes the binary logarithm.)
	We show a $\log\log$ improvement over this bound, as follows.
\begin{theorem}
	\label{thm:llp5}
	There exists $d\ge1$ such that every $P_5$-free graph $G$ with clique number at most $w\ge3$ has chromatic number at most $w^{d\log w/\log\log w}$.
\end{theorem}
	The proof method is via the following \dd polynomial versus linear complete pairs\ee{} fact:
\begin{theorem}
	\label{thm:mainll}
	There exists $b\ge3$ such that for every $P_5$-free graph $G$ with clique number $w\ge2$, there is a complete pair $(A,B)$ in $G$ with $\chi(A)\ge w^{-b}\chi(G)$ and $\chi(B)\ge 2^{-b}\chi(G)$.
\end{theorem}
	Since \cref{thm:gs} implies that every $P_5$-free graph contains a vertex whose neighbourhood has linear chromatic number, this result can be viewed as a corollary of \cref{conj:pgs} for $P_5$.
	Such a \dd polynomial versus linear\ee{} form is inspired by a conjecture of Conlon, Fox, and Sudakov~\cite{MR3497267} that for every graph $H$, there exists $d\ge1$ such that for every $\eps>0$ and every $H$-free graph $G$,
	there are disjoint $A,B\subset V(G)$ for which $\abs A\ge \eps^d\abs G$, $\abs B\ge 2^{-d}\abs G$, and either every vertex in $B$ has fewer than $\eps\abs A$ neighbours in $A$ or every vertex in $B$ has fewer than $\eps\abs A$ nonneighbours in $A$.
	Such a predicted configuration was an important step in the proof of the currently best known bound towards the \erh{} conjecture~\cite{density1}.
	Also, we would like to remark that the \gas{} conjecture \ref{conj:gs} is equivalent to the following \dd complete pairs\ee{} statement, which might possibly be useful in the study of the conjecture.
\begin{conjecture}
	[\gas{}]
	For every $\ell,w\ge1$ and every forest $T$, there exists $k\ge2$ such that every $T$-free graph $G$ with $\chi(G)\ge k$ and $\omega(G)\le w$ contains a complete pair $(A,B)$ with $\chi(A),\chi(B)\ge \ell$.
\end{conjecture}
(The proof of equivalence can be done by induction on $w$ and we omit it.)

	Let us see how \cref{thm:mainll} gives \cref{thm:llp5}.
	In what follows, 
	a {\em blockade} in a graph $G$ is a sequence $(B_1,\ldots,B_k)$ of disjoint (and possibly empty) subsets of $V(G)$, where each $B_i$ is a {\em block} of the blockade;
	and this blockade is {\em complete} in $G$ if $B_i$ is complete to $B_j$ for all distinct $i,j\in[k]$.
\begin{proof}
	[Proof of \cref{thm:llp5}, assuming \cref{thm:mainll}]
	Let $b\ge4$ be given by \cref{thm:mainll}.
	We claim that $d:=2b$ suffices.
	To see this,
	let $f(w):=w^{d\log w/\log\log w}$ for all $w\ge3$.
	We will prove by induction on $w\ge3$ that $\chi(G)\le f(w)$ for
	every $P_5$-free graph $G$ with clique number at most $w\ge 3$.
	If $w\le 16$ then $\chi(G)\le 3^w\le w^8\le w^d\le f(w)$ by \cref{thm:gs} and the choice of $d$; so we may assume $w\ge 16$.
	
	Let $k\ge0$ be maximal such that there is a complete blockade $(B_0,B_1,\ldots,B_k)$ in $G$
	with $\chi(B_k)\ge 2^{-bk}\chi(G)$ and $\chi(B_{i-1})\ge w^{-2b}\chi(G)$ for all $i\in[k]$;
	such a $k$ exists since this is satisfied for $k=0$ with $B_0=V(G)$.
	If $k<\log w$,
	then $\chi(B_k)\ge 2^{-bk}\chi(G)\ge w^{-b}\chi(G)\ge w^{d-b}\ge2$.
	The choice of $b$
	yields a complete pair $(A,B)$ in $G[B_k]$ with $\chi(A)\ge w^{-b}\chi(B_k)\ge w^{-2b}\chi(G)$
	and $\chi(B)\ge 2^{-b}\chi(B_k)\ge 2^{-b(k+1)}\chi(G)$.
	Hence the blockade $(B_0,B_1,\ldots,B_{k-1},A,B)$ violates the maximality of $k$.
	Therefore $k\ge \log w$;
	and so there exists $i\in\{0,1,\ldots,k-1\}$ such that $G[B_i]$ has clique number at most $w/\log w$.
	Let $y:=w/\log w\in[4,w)$ (note that $w\ge16$).
	Since the function $x\mapsto \log x/\log\log x$ is increasing on $[4,\infty)$, we see that
	\begin{align*}
		\log(f(w))-\log(f(y))
		&=d\left(\frac{(\log w)^2}{\log\log w}-\frac{(\log y)^2}{\log\log y}\right)\\
		&\ge d\left(\frac{(\log w)^2}{\log\log w}-\frac{\log w\log y}{\log\log w}\right)
		= \frac{d\log w\log\frac wy}{\log\log w}
		=d\log w.
	\end{align*}
	Thus, by the choice of $d$ and by induction, we obtain
	$\chi(G)\le w^{2b}\chi(B_i)\le w^{2b}f(y)\le w^{2b-d}f(w)\le f(w)$.
	This proves \cref{thm:llp5}.
\end{proof}
	Next, we say that a forest $T$ is {\em poly-$\chi$-bounding} if it satisfies \cref{conj:pgs}.
	Given the undecidability of this property for $P_5$,
	it is natural to ask whether it holds for all $P_5$-free forests $T$.
	Scott, Seymour, and Spirkl~\cite{MR4472775} did this for all $P_5$-free {\em trees}; these are the {\em double stars} which are graphs obtained from $P_4$ by substituting each leaf by an arbitrary edgeless graph.
	The case of general $P_5$-free forests -- disjoint unions of double stars -- remains open, because the poly-$\chi$-bounding property is not known to be closed under disjoint unions (on the other hand, it is not hard to show that the family of forests satisfying \cref{conj:gs} has this property).
	We say that a tree $T$ is {\em addible} if the following holds: if $T$ is poly-$\chi$-bounding, then for every poly-$\chi$-bounding forest $J$, the disjoint union of $T$ and $J$ is also poly-$\chi$-bounding.
	In order to prove that every $P_5$-free forest is poly-$\chi$-bounding,
	it remains to show that every double star is addible.
	Known partial results~\cite{MR4472774,MR4560487} in this direction include:
	\begin{theorem}
		[Scott--Seymour--Spirkl]
		\label{thm:starad}
		Every star is addible.
	\end{theorem}
	\begin{theorem}
		[Chudnovsky--Scott--Seymour--Spirkl]
		\label{thm:pathad}
		Every path is addible.
	\end{theorem}
	\cref{thm:starad} implies that every star forest is poly-$\chi$-bounding, and \cref{thm:pathad} particularly implies that every disjoint union of copies of $P_4$ is poly-$\chi$-bounding.
	
	In what follows, for integers $k,t\ge1$, a {\em $(k,t)$-broom} is the graph obtained by substituting a $t$-vertex edgeless graph  for a leaf of the $(k+1)$-vertex path;
	and a {\em $t$-broom} is a $(3,t)$-broom.
	The second goal of this paper is to unify \cref{thm:starad,thm:pathad} and in turns shows that all disjoint unions of $t$-brooms are poly-$\chi$-bounding, as follows:
	\begin{theorem}
		\label{thm:broom}
		For all integers $k,t\ge1$, the $(k,t)$-broom is addible.
	\end{theorem}
	
	The method of proof of \cref{thm:broom} is via the following, which particularly implies that every $(k,t)$-broom-free graph contains polynomially high-chromatic anticomplete pairs:
	\begin{theorem}
		\label{thm:anti}
		For every $k,t\ge1$, there exists $d\ge1$ such that every non-complete graph $G$ with clique number $w\ge2$ contains either:
		\begin{itemize}
			\item an anticomplete pair $(A,B)$ with $\chi(A),\chi(B)\ge w^{-d}\chi(G)$; or
			
			\item an anticomplete pair $(P,Q)$ where $G[P]$ is a $(k,t)$-broom and $\chi(Q)\ge w^{-d}\chi(G)$.
		\end{itemize}
	\end{theorem}
	This high-chromatic anticomplete pairs result is inspired by a conjecture of El-Zahar and Erd\H os~\cite{MR845138} that says graphs with huge chromatic number and bounded clique number contains a high-chromatic anticomplete pair
	(but not necessarily linear in  the chromatic number of the graphs in question; see~\cite{MR4676642} for some partial results on this problem),
	and a result of Liebenau, Pilipczuk, Seymour, and Spirkl~\cite{MR3926277} 
	on such pairs in graphs with no induced {\em caterpillar} (with exponential dependence on the clique number),
	which is a tree obtained from a path by joining new vertices of degree one to the vertices on the path.
	Since every $(k,t)$-broom is a caterpillar, it could be true that \cref{thm:anti} holds for caterpillars, and more generally for all forests.

	The rest of the paper is organised as follows.
	In \cref{sec:warmup}, we provide arguments and results which would help explain the ideas and methods presented in later parts of the paper; along the way, we obtain new proofs of several known results in the literature.
	Then we present the proof of \cref{thm:anti} in \cref{sec:broom},
	and the proof of \cref{thm:mainll} in \cref{sec:llp5}.
	We remark that the proof of \cref{thm:mainll} uses \cref{thm:ehp5} and the $P_5$ case of \cref{thm:anti}.
\section{Some expository arguments}
\label{sec:warmup}
\subsection{Excluding a $t$-broom}
\label{sec:tbroom}
To illustrate some ideas employed in the rest of this paper, let us give a short proof of the polynomial $\chi$-boundedness of the class of $t$-broom-free graphs,
which was first proved by Liu, Schroeder, Wang, and Yu~\cite{MR4589655}
via a variant of the \dd template\ee{} method introduced by Gy\'arf\'as, Szemer\'edi, and Tuza~\cite{MR573638}.
This method was used by Kierstead and Penrice~\cite{MR1258244} to prove the \gas{} conjecture \ref{conj:gs} for trees of radius two,
and adapted by Scott, Seymour, and Spirkl~\cite{MR4472775} to show that all double stars are poly-$\chi$-bounding.
The argument in~\cite{MR4589655} gives the $\chi$-binding function $Cw^2R(t,w)$ where $C>t^2$ depends on $t$ only and $R(t,w)$ is the least integer $n\ge1$ such that every $n$-vertex graph has a stable set of size $t$ or a clique of size $w$ (the standard Ramsey number).
Building on high-chromatic anticomplete and \dd near-complete\ee{} pairs, our proof yields the explicit $\chi$-binding function $2w^2R(t,w)$.
We begin with:
\begin{lemma}
\label{lem:tbroom}
Every $t$-broom-free graph $G$ with clique number $w\ge2$ has nonempty disjoint $S,P\subset V(G)$ with
$\chi(G)\le\chi(S)+\chi(P)$,
$\chi(P)\ge w^{-1}\chi(G)$,
and $\chi(P\setminus N_G(u))\le 2R(t,w)-1$ for all $u\in S$.
\end{lemma}
\begin{proof}
We may assume $G$ is connected.
If $G$ is complete, then we are done by taking $S,P$ as two distinct singletons (this is doable since $w\ge2$).
Thus we may also assume $G$ is non-complete.
Let $K$ be a maximum clique in $G$;
and for every $v\in K$, let $P_v:=V(G)\setminus(K\cup N_G(v))$.
Then $\bigcup_{v\in K}P_v=V(G)\setminus K$, which gives $v\in K$ with 
$$\chi(P_v\cup\{v\})\ge \abs K^{-1}\chi(G)\ge w^{-1}\chi(G)>1=\chi(\{v\}).$$
Hence $\chi(P_v)=\chi(P_v\cup\{v\})\ge w^{-1}\chi(G)$.
By taking a component of $P_v$ with chromatic number at least $w^{-1}\chi(G)$,
we obtain an anticomplete pair $(P,Q)$ in $G$ such that $G[P],G[Q]$ are connected, $\chi(P)\ge w^{-1}\chi(G)$, and $\chi(P)\ge\chi(Q)\ge1$.
Among all such pairs $(P,Q)$, choose $P,Q$ with $\chi(P)+\chi(Q)$ is maximal; and subject to this, with $\abs P+\abs Q$ maximal.
Since $G$ is connected, there is a minimal nonempty cutset $S$ separating $P,Q$ in $G$.
By the maximality of $(P,Q)$,
$G[P],G[Q]$ are components of $G\setminus S$, every vertex in $S$ has a neighbour in each of $P,Q$,
and $\chi(G\setminus S)=\chi(P)$.
Thus
$\chi(G)\le \chi(S)+\chi(G\setminus S)=\chi(S)+\chi(P)$.
(We remark that this type of argument will appear frequently in \cref{sec:llp5}.)

In what follows, the {\em degeneracy} of a graph $G$ is the least integer $d\ge0$ for which there is an ordering $(v_1,\ldots,v_n)$ of $V(G)$ such that for all $i\in[n]$, $v_i$ has at most $d$ neighbours in $\{v_{i+1},\ldots,v_n\}$;
in other words, the degeneracy of $G$ is the least integer $d\ge0$ for which every induced subgraph of $G$ has minimum degree at most $d$.
By greedy colouring, the degeneracy of $G$ is at least $\chi(G)-1$.
Now we use the $t$-broom-freeness of $G$ to show that $S$ is \dd near-complete\ee{} to $P$:
\begin{claim}
	\label{claim:tbroom}
	For each $u\in S$, $G[P\setminus N_G(u)]$ has degeneracy at most $2(R(t,w)-1)$.
\end{claim}
\begin{subproof}
	Suppose not; then there is an induced subgraph $F$ of $G[P\setminus N_G(u)]$ with minimum degree at least $2R(t,w)-1$.
	Since $G[P]$ is connected and $N_G(u)$ is nonempty,
	there is a shortest path $v_1\text-\cdots\text-v_k$ from $N_G(u)$ to $V(F)$ in $P$;
	in particular $v_1\in N_G(u)$, $v_k\in V(F)$, and none of $v_1,\ldots,v_{k-2}$ has a neighbour in $V(F)$.
	Note that $\omega(N_F(v_k))<w$.
	Let $z\in N_G(u)\cap Q$.
	If $v_{k-1}$ has at least $R(t,w)$ neighbours in $N_F(v_k)$,
	then it is complete to a stable set $T\subset N_F(v_k)$ with $\abs T=t$;
	and so
	$\{z,u,v_1,\ldots,v_{k-1}\}\cup T$ would induce a $(k+1,3)$-broom in $G$, a contradiction since $k\ge2$.
	Thus $v_{k-1}$ has at least $d_F(u)-(R(t,w)-1)\ge R(t,w)$ nonneighbours in $N_F(v_k)$;
	and so $v_{k-1}$ is anticomplete to a stable set $T\subset N_F(v_k)$ with $\abs T=t$.
	But then $\{z,u,v_1,\ldots,v_k\}\cup T$ would induce a $(k+2,3)$-broom in $G$, a contradiction.
	This proves \cref{claim:tbroom}.
\end{subproof}
\cref{claim:tbroom} yields $\chi(P\setminus N_G(u))\le 2R(t,w)-1$ for all $u\in S$.
This completes the proof of \cref{lem:tbroom}.
\end{proof}
We now show that every $t$-broom is poly-$\chi$-bounding with $\chi$-binding function $2w^2R(t,w)$:
\begin{theorem}
\label{thm:tbroom}
Every $t$-broom-free graph $G$ with clique number $w\ge1$ satisfies $\chi(G)\le 2w^2R(t,w)$.
\end{theorem}
\begin{proof}
We proceed by induction on $\abs G$.
We may assume $w\ge2$.
By \cref{lem:tbroom},
there are nonempty disjoint $S,P\subset V(G)$ with
$\chi(G)\le\chi(S)+\chi(P)$,
$\chi(P)\ge w^{-1}\chi(G)$,
and $\chi(P\setminus N_G(u))\le 2R(t,w)-1$ for all $u\in S$.
Let $C$ be a maximum clique in $G[S]$ and $q:=\abs C\ge1$.
Then $\chi(S)\le 2q^2R(t,q)\le 2q^2R(t,w)$ by induction.
Let $D$ be the set of vertices in $P$ with a nonneighbour in $C$;
then
$\chi(D)\le q(2R(t,w)-1)$ by \cref{claim:tbroom}.
If $P=D$ then
$w^{-1}\chi(G)\le\chi(P)=\chi(D)\le 2wR(t,w)$ and so $\chi(G)\le 2w^2R(t,w)$.
Thus we may assume $D\subsetneq P$.
Since $P\setminus D$ is complete to $C$,
we have $1\le\omega(P\setminus D)\le w-\abs C=w-q$;
and so $1\le q\le w-1$ and $\chi(P\setminus D)\le 2(w-q)^2R(t,w-q)\le 2(w-q)^2R(t,w)$ by induction.
Therefore
\[\begin{aligned}
	\chi(G)\le \chi(S)+\chi(P)
	&\le \chi(S)+\chi(P\setminus D)+\chi(D)\\
	&\le 2q^2R(t,w)+2(w-q)^2R(t,w)+q(2R(t,w)-1)\\
	&\le 2((w-1)^2+1)R(t,w)+(w-1)(2R(t,w)-1)\le 2w^2R(t,w)
\end{aligned}\]
where the penultimate inequality holds since $q^2+(w-q)^2\le(w-1)^2+1$ and $1\le q\le w-1$.
This completes the induction step and proves \cref{thm:tbroom}.
\end{proof}
\subsection{A general result}
In this subsection we will prove the following \dd quasi-polynomially high-chromatic and near-anticomplete pairs\ee{} result for excluding a general induced subgraph:
\begin{theorem}
	\label{thm:quasichi}
	For every graph $H$, every non-complete $H$-free graph $G$ with $w:=\omega(G)$ contains disjoint $A,B\subset V(G)$ such that $\chi(A),\chi(B)\ge w^{-2\abs H\log w}\chi(G)$ and $\chi(A\cap N_G(v))<w^{-1}\chi(A)$ for all $v\in B$.
\end{theorem}
We remark that this \dd near-anticomplete\ee{} property, in general, cannot be turned to a full anticomplete one.
For instance, Raphael Steiner (private communication) observed that when $H$ is the triangle, one can consider the triangle-free process analyzed by Bohman~\cite{MR2522430} to obtain an $n$-vertex triangle-free graph with chromatic number at least $\Omega(\sqrt{n/\log n})$ and no anticomplete pairs of size at least $\Omega(n)$ (we omit the detailed calculations),
and so with no linear-chromatic anticomplete pairs since Ajtai, Koml\'os, and Szemer\'edi~\cite{MR600598} proved that every $m$-vertex triangle-free graph has chromatic number at most $O(\sqrt{m/\log m})$.
However, it could be true that a full anticomplete outcome holds when $H$ is a forest (see~\cite{MR4170220} for a related result on linear-sized anticomplete pairs), and in that case the term $w^{-2\abs H\log w}$ can be turned into $\operatorname{poly}(w^{-1})$.

\cref{thm:quasichi} can be used to deduce the fact that \dd near-Esperet\ee{} graphs are closed under disjoint union (we omit the proof),
where a graph $H$ is {\em near-Esperet} (defined in~\cite{MR4809068}) if there exists $d>0$ such that $\chi(G)\le w^{d\log w}$ for all $H$-free graphs $G$ with clique number $w$.
In fact, our argument in this subsection is robust enough to give a new proof (yielding similar bounds) of~\cite[Lemma 4.2]{MR4809068}, which immediately gave all of the results in~\cite[Section 4]{MR4809068}.
To do so, we need a couple more definitions.
A {\em copy} of a graph $H$ in a graph $G$ is an injective map $\varphi\colon V(H)\to V(G)$ such that for all distinct $u,v\in V(H)$, $uv\in E(H)$ if and only if $\varphi(u)\varphi(v)\in E(G)$.
A {\em submeasure} on $G$ is a function $\mu\colon2^{V(G)}\to\mab R^+$ satisfying:
\begin{itemize}
	\item $\mu(\emptyset)=0$ and $\mu(\{v\})=1$ for all $v\in V(G)$;
	
	\item $\mu(X)\le\mu(Y)$ for all $X\subset Y\subset V(G)$ (monotonicity); and
	
	\item $\mu(A\cup B)\le \mu(A)+\mu(B)$ for all $A,B\subset V(G)$ (subadditivity).
\end{itemize}
 
For example, the chromatic number $\chi(G)$ and the number of vertices $\abs G$ are submeasures on $G$.
Submeasures appeared in the work of Liebenau, Pilipczuk, Seymour, and Spirkl~\cite{MR3926277} (in a \dd normalised\ee{} form and under the name \dd measures\ee{}) where they proved that if $G$ has no copy of a given caterpillar and $\mu$ is a submeasure on $G$,
then $G$ admits an anticomplete pair $(A,B)$ with $\mu(A),\mu(B)\ge \mu(G)/2^{O(\omega(G))}$.
Here we will discuss another application of submeasures in $\chi$-boundedness.
In what follows, for a graph $G$ and $X,Y,A,B\subset V(G)$, we write $(X,Y)\subset(A,B)$ if $X\subset A$ and $Y\subset B$.
Let us reconstruct an argument of Erd\H os and Hajnal~\cite{MR1031262} to deduce the following:
\begin{lemma}
	[\erh{}]
\label{lem:gen}
Let $\eps\in(0,\frac12]$, and let $H$ be a graph with $V(H)=\{g_1,\ldots,g_h\}$.
Let $G$ be a graph, and let $A_1,\ldots,A_h\subset V(G)$ be nonempty and disjoint.
Then for every submeasure $\mu$ on $G$, either:
\begin{itemize}
	\item there is a copy $\varphi\colon V(H)\to V(G)$ with $\varphi(g_i)\in A_i$ for all $i\in[h]$; or
	
	\item there are $i,j\in[h]$ with $i<j$, and $(D_i,D_j)\subset (A_i, A_j)$ such that $\mu(D_i)\ge \eps^{h-2}\mu(A_i)$, $\mu(D_j)\ge\eps^{h-2}\mu(A_j)$, and one of the following holds:
	\begin{itemize}
		\item $\mu(D_j\cap N_G(v))<\eps\cdot\mu(D_j)$ for all $v\in D_i$ ($D_i$ is \dd$\eps$-sparse\ee{} to $D_j$); and
		
		\item $\mu(D_j\setminus N_G(v))<\eps\cdot\mu(D_j)$ for all $v\in D_i$ ($D_i$ is \dd$(1-\eps)$-dense\ee{} to $D_j$).
	\end{itemize}
\end{itemize}
\end{lemma}
\begin{proof}
The lemma is true for $h\le2$.
Let us prove it for $h\ge3$, assuming that is true for $H\setminus g_1$.
To this end, we may assume that the second outcome of the theorem does not hold.
For every $i\in[h]\setminus\{1\}$,
let $B_i$ be the set of vertices $v\in A_1$ with $\mu(N_G(v)\cap A_i)<\eps\cdot\mu(A_i)$ (if $g_1g_i\in E(H)$)
or the set of vertices $v\in A_1$ with $\mu(A_i\setminus N_G(v))<\eps\cdot\mu(A_i)$ (if $g_1g_i\nin E(H)$).
Since the second outcome of the lemma fails,
$\mu(B_i)<\eps^{h-2}\mu(A_1)$ for all $i\in[h]\setminus\{1\}$.
Then by subadditivity and since $(h-1)\eps^{h-2}\le (h-1)2^{2-h}\le1$, 
$$\mu(B_2\cup\cdots\cup B_h)<(h-1)\eps^{h-2}\mu(A_1)\le\mu(A_1).$$
Thus there exists $v\in A_1\setminus(B_2\cup\cdots\cup B_h)$.
For every $i\in[h]\setminus\{1\}$,
let $C_i:=N_G(v)\cap A_i$ (if $g_1g_i\in E(H)$) or $C_i:=A_i\setminus N_G(v)$ (if $g_1g_i\nin E(H))$;
then $\mu(C_i)\ge \eps\cdot \mu(A_i)$.
Since the second outcome of the lemma fails,
there are no $i,j\in[h]\setminus\{1\}$ with $i<j$ and $(D_i,D_j)\subset(C_i,C_j)$ such that
\begin{itemize}
	\item $\mu(D_i)\ge \eps^{h-3}\mu(C_i)$, $\mu(D_j)\ge\eps^{h-3}\mu(C_j)$; and
	
	\item either $\mu(D_j\cap N_G(v))<\eps\cdot \mu(D_j)$ for all $v\in D_i$ or $\mu(D_j\setminus N_G(v))<\eps\cdot \mu(D_j)$ for all $v\in D_i$.
\end{itemize}
Hence by induction,
there is a copy $\varphi$ of $H\setminus g_1$ in $G$ with $\varphi(g_i)\in C_i$ for all $i\in[h]\setminus\{1\}$.
Extending $\varphi$ by defining $\varphi(g_1):=v$ completes the induction step.
This completes the proof of \cref{lem:gen}.
\end{proof}
	In what follows, let $\mu(G):=\mu(V(G))$ for every graph $G$ and every submeasure $\mu$ on $G$.
	From \cref{lem:gen}, we obtain the following \dd near-pure pairs of polynomial submeasure\ee{} result:
\begin{lemma}
	\label{thm:gen}
	For every $\eps\in(0,\frac12]$ and every graph $H$, every $H$-free graph $G$, and every submeasure $\mu$ on $G$,
	there are disjoint $A,B\subset V(G)$ such that:
	\begin{itemize}
		\item $\mu(A),\mu(B)\ge (2\abs H)^{-1}\eps^{\abs H-2}\mu(G)$; and
		
		\item either $\mu(A\cap N_G(v))<\eps\cdot \mu(A)$ for all $v\in B$ or $\mu(A\setminus N_G(v))<\eps\cdot\mu(A)$ for all $v\in B$.
	\end{itemize} 
\end{lemma}
\begin{proof}
	Let $h:=\abs H$.
	We may assume $h\ge 2$ and $\mu(G)\ge 2h\cdot\eps^{2-h}$, for otherwise the lemma trivially holds.
	Now, let $\ell\ge0$ be maximal such that there are disjoint $A_1,\ldots,A_\ell\subset V(G)$ with $(2h)^{-1}\mu(G)<\mu(A_i)\le h^{-1}\mu(G)$ for all $i\in[\ell]$.
	Let $S:=A_1\cup\cdots\cup A_\ell$.
	If $\ell<h$, then
	$\mu(S)\le \ell\cdot h^{-1}\mu(G)\le(1-h^{-1})\mu(G)$
	and so $\mu(G\setminus S)\ge h^{-1}\mu(G)$.
	Hence there exists minimal $A_{\ell+1}\subset V(G)\setminus S$ with $\mu(A_{\ell+1})\ge (2h)^{-1}\mu(G)$.
	For every $v\in A_{\ell+1}$, the minimality of $A_{\ell+1}$ yields
	$\mu(A_{\ell+1})\le \mu(A_{\ell+1}\setminus\{v\})+\mu(\{v\})
	<(2h)^{-1}\mu(G)+1\le h^{-1}\mu(G)$
	where the last inequality holds since $\mu(G)\ge2h$.
	Thus $A_1,\ldots,A_\ell,A_{\ell+1}$ violate the maximality of $\ell$.
	This shows that $\ell\ge h$; and so $A_1,\ldots,A_h$ are defined.
	
	Now, by \cref{lem:gen}, there are $i,j\in[h]$ with $i<j$, and $(D_i,D_j)\subset (A_i,A_j)$ such that:
	\begin{itemize}
		\item $\mu(D_i)\ge \eps^{h-2}\mu(A_i)\ge (2h)^{-1}\eps^{h-2}\mu(G)$
		and $\mu(D_j)\ge \eps^{h-2}\mu(G)\ge(2h)^{-1}\eps^{h-2}\mu(G)$; and
		
		\item either $\mu(D_j\cap N_G(v))<\eps\cdot \mu(D_j)$ for all $v\in D_i$ or $\mu(D_j\setminus N_G(v))<\eps\cdot \mu(D_j)$ for all $v\in D_i$.
	\end{itemize}
	This proves \cref{thm:gen}.
\end{proof}
	It is not hard to deduce~\cite[Lemma 4.2]{MR4809068} from the above lemma (with similar bounds) and in turn reprove~\cite[Theorem 2.1]{MR3194194}.
	To see this, in the proof of~\cite[Lemma 4.2]{MR4809068} we can let $\eps:=k^{-1}$, $H$ be the disjoint union of some member of $\mac H_1$ and some member of $\mac H_2$, and $\mu$ be the submeasure defined there (we omit the details).
	Now, we use \cref{thm:gen} to prove the following result, which yields \cref{thm:quasichi} with $\mu=\chi$ and $\eps=w^{-1}$:
\begin{theorem}
	\label{thm:quasi}
	Let $H$ be a graph, let $G$ be an $H$-free graph, and let $w:=\omega(G)$.
	Then for every $\eps\in(0,w^{-1}]$ and every submeasure $\mu$ on $G$, either:
	\begin{itemize}
		\item $G$ has a stable set $S$ with $\mu(S)\ge \eps^{2\abs H\log w}\mu(G)$; or
		
		\item there are disjoint $A,B\subset V(G)$ with $\mu(A),\mu(B)\ge \eps^{2\abs H\log w}\mu(G)$ and $\mu(A\cap N_G(v))<\eps\cdot\mu(A)$ for all $v\in B$.
	\end{itemize}
\end{theorem}
\begin{proof}
	The main idea is to iterate the \dd near-complete\ee{} outcome in \cref{thm:gen} to get a long sequence of disjoint vertex subsets of $G$ whose sum of clique numbers is small while each of them has sufficiently large submeasure.
	To carry this out, assume that the first outcome does not hold; and so $w\ge2$, for otherwise we could take $S=V(G)$.
	Now, let $\ell\ge1$ be maximal such that $\ell$ is a power of two and there are disjoint $E_1,\ldots,E_{\ell}\subset V(G)$ satisfying:
	\begin{itemize}
		\item $\omega(E_1)+\cdots+\omega(E_{\ell})\le w$; and
		
		\item $\mu(E_i)\ge \eps^{2\abs H\log\ell}\mu(G)$ for all $i\in[\ell]$.
	\end{itemize}
	(Such an $\ell$ exist since these are satisfied for $\ell=1$, taking $E_1=V(G)$.)
	For each $i\in[\ell]$, the second bullet gives $\omega(E_i)\ge1$; and so the first bullet gives $\ell\ge w$. 
	But since the first outcome of the theorem does not hold, $\omega(E_i)\ge2$ for all $i\in[\ell]$; and so the first bullet yields $2\ell\le w$.
	For each $i\in[\ell]$, by \cref{thm:gen} (with $\eps=w^{-1}$ and $G[E_i]$ in place of $G$), there are disjoint $D_{2i-1},D_{2i}\subset E_i$ such that:
	\begin{itemize}
		\item $\mu(D_{2i-1}),\mu(D_{2i})\ge (2\abs H)^{-1}\eps^{\abs H-2}\mu(E_i)\ge \eps^{2\abs H-2}\mu(E_i)\ge \eps^{2\abs H\log (2\ell)-2}\mu(G)$; or
		
		\item either $\mu(D_{2i}\cap N_G(v))<\eps\cdot\mu(D_{2i})$ for all $v\in D_{2i-1}$
		or $\mu(D_{2i}\setminus N_G(v))<\eps\cdot\mu(D_{2i})$ for all $v\in D_{2i-1}$.
	\end{itemize}
	
	If $\mu(D_{2i}\cap N_G(v))<\eps\cdot\mu(D_{2i})$ for all $i\in D_{2i-1}$ then the second outcome of the theorem holds and we are done.
	Thus we may assume $\mu(D_{2i}\setminus N_G(v))<\eps\cdot\mu(D_{2i})$ for all $i\in D_{2i-1}$.
	Now, let $D$ be a largest clique in $G[D_{2i-1}]$, and let $D_{2i}'$ be the set of vertices in $D_{2i}$ complete to $D$; then $\omega(D_{2i-1})+\omega(D_{2i}')=\abs D+\omega(D_{2i}')\le\omega(E_i)$
	and $\mu(D_{2i}')>(1-\abs D\cdot\eps)\mu(D_{2i})\ge(1-\abs D\cdot w^{-1})\mu(D_{2i})\ge0$.
	Hence $\omega(D_{2i}')\ge 1$; and so $\abs D\le\omega(E_i)-1\le w-1$.
	It follows that
	$$\mu(D_{2i}')\ge (1-\abs D\cdot \eps)\mu(D_{2i})
	\ge (1-\abs D\cdot w^{-1})\mu(D_{2i})\ge w^{-1}\mu(D_{2i})\ge \eps^{2\abs H\log(2\ell)}\mu(G).$$
	
	Since $D_{2i-1},D_{2i}'$ are defined for all $i\in[\ell]$ and 
	$$\sum_{i=1}^\ell(\omega(D_{2i-1})+\omega(D_{2i}'))
	\le \sum_{i=1}^\ell\omega(E_i)\le w,$$ 
	the collection $(D_{2i-1},D_{2i}':i\in[\ell])$ then contradicts the maximality of $\ell$.
	This proves \cref{thm:quasi}.
\end{proof}
	In the case $\mu=\chi$, it would be interesting to extend the range of $\eps$ in \cref{thm:quasi} to $(0,c)$ for some $c>0$ depending on $H$ only, since this would imply that every forest is addible and so is poly-$\chi$-bounding. 
	A particular goal of the upcoming section is to show that we can do even better than this (polynomially high-chromatic anticomplete pairs) when $H$ is a disjoint union of brooms.
\section{New confirmed cases of polynomial \gas{}}
\subsection{Adding a path via covering blockades}
This section provides a proof of \cref{thm:anti} in the special case with paths in place of brooms,
which is an adaptation of the Gy\'arf\'as path argument \ref{thm:gs} and gives a new proof of \cref{thm:pathad}.
In this previous section, the proof of \cref{thm:tbroom} was done by obtaining a \dd near-complete\ee{} pair of vertex subsets which together \dd occupy\ee{} all of the chromatic number of the host graph.
In many situations, one can iterate this outcome inside the \dd big\ee{} subset with highest chromatic number each time to obtain a long blockade where each block is near-complete to each previous one.
The following definition formulates this idea.
For $\eps>0$, a blockade $(B_1,\ldots,B_k)$ in a graph $G$ is {\em $\eps$-vivid} if for all $i,j\in[k]$ with $i<j$ and every $v\in B_j$,
$\chi(B_i\setminus N_G(v))<\eps\cdot\chi(B_i)$.
This is an analogue of dense or sparse blockades used frequently in the recent work on the \erh{} conjecture;
and the following lemma shows that $\omega(G)^{-1}$-vivid blockades cannot be too long.
\begin{lemma}
	\label{lem:iterate}
	Let $G$ be a graph with clique number at most $w\ge2$,
	and let $(B_1,\ldots,B_k)$ be a $w^{-1}$-vivid blockade in $G$.
	Then $k\le w$.
\end{lemma}
\begin{proof}
	Suppose not.
	Let $\ell\ge1$ be maximal such that $G$ has a clique $K=\{v_1,\ldots,v_\ell\}$ with $v_{i}\in B_{k-i+1}$ for all $i\in [\ell]$.
	Then $\ell\le w<k$;
	and so for $j:=k-\ell\ge1$,
	we have that
	$\chi(B_j\setminus \bigcup_{i\in [\ell]}N_G(v_{i}))<\ell\cdot w^{-1}\chi(B_j)
	\le \chi(B_j)$.
	Hence there would be $v_{\ell}\in B_j$ complete to $K$, contrary to the maximality of $\ell$.
	This proves \cref{lem:iterate}.
\end{proof}
We next introduce the central objects in this section.
For $k\ge1$, and for a graph $G$ with clique number $w\ge2$,
a {\em $k$-covering blockade} in $G$ is a blockade $(D_1,\ldots,D_k,E)$ of nonempty disjoint subsets of $V(G)$ such that:
\begin{itemize}
	\item for every $i\in[k]$, every vertex in $D_i$ has a neighbour in $D_{i-1}$ and no neighbour in $D_1\cup\cdots\cup D_{i-2}$;
	
	\item $E$ is anticomplete to $D_1\cup\cdots\cup D_{k-1}$; and
	
	\item for every 
	$X\subset D_k$ and $Y\subset E$ with $\chi(Y)\ge w^{-3}\chi(E)$,
	the set of vertices $u\in X$ with $\chi(Y\setminus N_G(u))< w^{-1}\chi(Y)$ has chromatic number less than $(1-w^{-2})\chi(X)$.
\end{itemize}

The existence of $1$-covering blockades with decent chromatic number is given by the following lemma. 
\begin{lemma}
	\label{lem:base}
	For every non-complete graph $G$ with clique number at most $w\ge2$,
	there is a $1$-covering blockade $(D,E)$ in $G$ with $\chi(D),\chi(E)\ge w^{-6}\chi(G)$.
\end{lemma}
\begin{proof}
	We may assume $\chi(G)\ge w^6$.
	Let $\ell\ge0$ be maximal such that there is a $w^{-1}$-vivid blockade $(B_0,B_1,\ldots,B_\ell)$ of nonempty disjoint subsets of $V(G)$ with
	$\chi(B_{i-1})\ge w^{-8}\chi(G)$ for all $i\in[\ell]$ and $\chi(B_\ell)\ge (1-w^{-1})^{2\ell}\chi(G)$.
	\cref{lem:iterate} implies $\ell<w$;
	and so
	$\chi(B_\ell)\ge (1-w^{-1})^{\ell}\chi(G)\ge 2^{-2}\chi(G)>w^4$.
	Let $E\subset B_\ell$ be such that $w^{-4}\chi(B_{\ell})\le\chi(E)\le 2w^{-4}\chi(B_{\ell})$,
	and let $D:=B_\ell\setminus E$;
	then
	$$\chi(D)\ge (1-w^{-2})\chi(B_\ell)
	\ge(1-w^{-1})^{\ell+1}\chi(G)\ge w^{-6}\chi(G).$$
	Let $X\subset D$ and $Y\subset E$ be such that  $\chi(Y)\ge w^{-3}\chi(E)\ge w^{-8}\chi(G)$ and
	$$\chi(X)\ge (1-w^{-2})\chi(D)
	\ge (1-w^{-2})^2\chi(B_\ell)
	\ge (1-w^{-1})\chi(B_\ell)\ge(1-w^{-1})^{\ell+1}\chi(G).$$
	If $X$ is $w^{-1}$-vivid to $Y$,
	then $(B_0,B_1,\ldots,B_{\ell-1},Y,X)$ would violate the maximality of~$k$.
	Thus there exists $v\in X$ with $\chi(Y\setminus N_G(v))\ge w^{-1}\chi(Y)$.
	This completes the proof of \cref{lem:base}.
\end{proof}
Since every $k$-covering blockade in $G$ yields an induced $P_{k+1}$ in $G$,
the following lemma immediately implies the case $t=1$ of \cref{thm:anti}, gives a new proof of \cref{thm:pathad}, and will be used to prove the general case in Subsection~\ref{sec:broom}.
We remark that the following argument is somewhat similar to the one in~\cite{density5}.
\begin{lemma}
	\label{lem:covering}
	For every $k\ge1$, every non-complete graph $G$ with clique number $w\ge1$ contains one of the following:
	\begin{itemize}
		\item an anticomplete pair $(A,B)$ with $\chi(A),\chi(B)\ge w^{-8k}\chi(G)$; and
		
		\item a $k$-covering blockade $(D_1,\ldots,D_k,E)$ with $\chi(D_k),\chi(E)\ge w^{-6k}\chi(G)$.
	\end{itemize}
\end{lemma}
\begin{proof}
	We proceed by induction on $k\ge1$.
	For $k=1$ this is true by \cref{lem:base}.
	Now, assume that the lemma holds for $k$;
	let us show it for $k+1$.
	Assume that the first outcome of the lemma does not hold;
	then by induction,
	there is a $k$-covering blockade $(D_1,\ldots,D_k,E)$ in $G$ with $\chi(D_k),\chi(E)\ge w^{-6k}\chi(G)$.
	Let $\ell\ge0$ be maximal such that there is a $w^{-1}$-vivid blockade $(B_0,B_1,\ldots,B_\ell)$ of disjoint subsets of $E$ such that $\chi(B_{i-1})\ge w^{-6}\chi(E)$
	for all $i\in[\ell]$
	and $\chi(B_\ell)\ge (1-w^{-1})^\ell\chi(E)$.
	Then $\ell<w$ by \cref{lem:iterate};
	and so $\chi(B_\ell)\ge (1-\ell\cdot  w^{-1})\chi(E)\ge w^{-1}\chi(E)$.
	Since $(D_1,\ldots,D_k,E)$ is a $k$-covering blockade,
	there exists $v\in D_k$ with $\chi(B_\ell\setminus N_G(v))\ge w^{-1}\chi(B_\ell)$.
	Thus there exists $A\subset D_k$ maximal such that the set $B$ of vertices in $B_\ell$ with no neighbour in $A$ satisfies $\chi(B)\ge w^{-1}\chi(B_{\ell})\ge w^{-6k-1}\chi(G)$.
	Since the first outcome of the lemma does not hold,
	$\chi(A)\le w^{-8(k+1)}\chi(G)\le w^{-8}\chi(D_k)$.
	Hence $\chi(D_k\setminus A)\ge(1-w^{-2})\chi(D_k)$;
	and so there exists $u\in D_k\setminus A$ with $\chi(B\setminus N_G(u))\ge w^{-1}\chi(B)
	$
	by the definition of $k$-covering blockades.
	Let $D_k':=A\cup\{u\}$,
	let $D_{k+1}$ be the set of vertices in $B_\ell$ with a neighbour in $D_k'$,
	and let $E':=B_\ell\setminus D_{k+1}$.
	
	To finish the induction step, we shall prove that 
	$(D_1,\ldots,D_{k-1},D_k',D_{k+1},E')$ is a $(k+1)$-covering blockade with $\chi(D_{k+1}),\chi(E')\ge w^{-6k-6}\chi(G)$.
	To see this, note that $D_{k+1},E'\subset E$ are anticomplete to $D_1\cup\cdots\cup D_{k-1}$.
	Also, by definition, every vertex in $D_{k+1}$ has a neighbour in $D_k'$ and $E'$ is anticomplete to $D_k'$.
	The maximality of $A$ yields $\chi(E')<w^{-2}\chi(B_\ell)$;
	and so $$\chi(D_{k+1})>(1-w^{-2})\chi(B_{\ell})\ge w^{-3}\chi(E)\ge w^{-6k-3}\chi(G).$$
	The choice of $u$ implies
	$$\begin{aligned}
		\chi(E')=\chi(B_\ell\setminus D_{k+1})&=\chi(B\setminus N_G(u))\\
		&\ge w^{-1}\chi(B)\ge w^{-2}\chi(B_\ell)\ge w^{-3}\chi(E)\ge w^{-6k-3}\chi(G).
	\end{aligned}$$
	Now, let $X\subset D_{k+1}$ and $Y\subset E'$ satisfy
	\begin{align*}
		\chi(Y)&\ge w^{-3}\chi(E')\ge w^{-6}\chi(E)\ge w^{-6k-6}\chi(G),\\
		\chi(X)&\ge (1-w^{-2})\chi(D_{k+1})
		\ge (1-w^{-2})^2\chi(B_\ell)
		\ge (1-w^{-1})\chi(B_\ell)
		\ge (1-w^{-1})^{\ell+1}\chi(E).
	\end{align*}
	If $X$ is $w^{-1}$-vivid to $Y$,
	then $(B_0,B_1,\ldots,B_{\ell-1},Y,X)$ would contradict the maximality of $\ell$.
	Therefore, there exists $z\in X$ with $\chi(Y\setminus N_G(z))\ge w^{-1}\chi(Y)$.
	This completes the induction step and the proof of \cref{lem:covering}.
\end{proof}
\subsection{Controlled induced subgraphs}
In what follows, for $q\ge w\ge2$,
a graph $G$ with clique number at most $w$, is {\em $q$-controlled} if $G$ is connected and $\chi(N_G(v))<(1-q^{-2})\chi(G)$ for all $v\in V(G)$;
we will drop the prefix \dd{}$w$-\ee{} from \dd$w$-controlled\ee{} for brevity when there is no danger of ambiguity.
The purpose of this definition is to replace a common approach in $\chi$-boundedness that uses induction on the clique number to
deduce that the neighbourhood of every vertex has not too large chromatic number.
The following lemma shows that there is always a controlled induced subgraph in $G$ with chromatic number almost equal~to~$\chi(G)$, which will be important in the proof of \cref{thm:anti} in Subsection~\ref{sec:broom} and the proof of \cref{thm:mainll} in \cref{sec:llp5}.
\begin{lemma}
	\label{lem:control}
	For every $q\ge w\ge 2$, every graph $G$ with clique number at most $w$ has a $q$-controlled induced subgraph with chromatic number more than $(1-wq^{-2})\chi(G)$.
\end{lemma}
\begin{proof}
	Let $k\ge0$ be maximal such that there exist a clique $S$ in $G$ with $\abs S=k$ and an induced subgraph $F$ of $G\setminus S$ with $V(F)$ complete to $S$ in $G$ and $\chi(F)\ge(1-q^{-2})^k\chi(G)$;
	such a $k$ exists since these conditions are satisfied for $k=0$, taking $S$ empty and $F=G$.
	Then $k<w$, and so
	$$\chi(F)\ge (1-q^{-2})^k\chi(G)\ge (1-kq^{-2})\chi(G)> (1-wq^{-2})\chi(G).$$
	If there exists $v\in V(F)$ with $\chi(N_F(v))\ge(1-q^{-2})\chi(F)\ge(1-q^{-2})^{k+1}\chi(G)$,
	then taking $S':=S\cup\{v\}$ and $F':=F[N_F(v)]$ would contradict the maximality of $k$.
	Hence every component of $F$ with chromatic number $\chi(F)$ is a $q$-controlled induced subgraph of $G$.
	This completes the proof of \cref{lem:control}.
\end{proof}
\subsection{Adding a broom}
\label{sec:broom}
This section contains the proof of \cref{thm:anti}.
Let us start by reproducing the argument of Scott, Seymour, and Spirkl~\cite{MR4472774} that proved their star addition result \ref{thm:starad}.
\begin{lemma}
	[Scott--Seymour--Spirkl]
	\label{lem:star}
	Let $t\ge1$ and $w\ge2$ be integers, let $F$ be a graph with $\omega(F)\le w$.
	Assume that there exists $A\subset V(F)$ with $\abs A\ge w^{t+2}$,
	and let $B\subset V(F)\setminus A$.
	Then for every $q\ge1$, $F$ contains either:
	\begin{itemize}
		
		\item a pair $(X,Y)\subset(A,B)$ with $\omega(X)+\omega(Y)\le \omega(F)$,
		$\abs{A\setminus X}< w^{t+2}$,
		and $\chi(B\setminus Y)< q$; or
		
		\item an anticomplete pair $(P,Q)\subset (A,B)$ where
		$P$ is a stable set of size $t$ and $\chi(Q)\ge w^{-t(t+2)}q$.
	\end{itemize}
\end{lemma}
\begin{proof}
	Assume that the second outcome does not hold.
	Let $n:=w^{t+1}-1$.
	Since $\abs A\ge w^{t+2}>nw$, there are $n$ nonempty cliques $A_1,\ldots,A_n\subset A$ such that for every $j\in[n]$,
	$A_j$ is a maximum clique in $F[A\setminus(A_1\cup\cdots\cup A_{j-1})]$.
	Let $p:=\abs{A_n}\ge1$ and $X:=A\setminus(A_1\cup\cdots\cup A_n)$; then $\abs{A\setminus X}\le nw< w^{t+2}$ and 
	$\omega(X)\le p\le \omega(F)$.
	Let $Y$ be the set of vertices in $B$ with fewer than $w^t$ nonneighbours in $A\setminus X$.
	\begin{claim}
		\label{claim:brooms1}
		$\omega(Y)\le \omega(F)-p$;
		and so $\omega(X)+\omega(Y)\le\omega(F)$.
	\end{claim}
	\begin{subproof}
		Suppose not;
		then there is a clique $K\subset Y$ with $\abs K>\omega(F)-p$.
		The number of vertices $A\setminus X$ with a nonneighbour in $K$ is at most $\abs K(w^t-1)< w^{t+1}-1=n$;
		and so there exists $j\in[n]$ such that $A_j$ is complete to $K$.
		By the definition of $A_j$, we have $\abs{A_j}\ge\abs{A_n}=p$;
		and thus $\omega(F)\ge\abs{A_j}+\abs K>p+\omega(F)-p=\omega(F)$, a contradiction.
		This proves \cref{claim:brooms1}.
	\end{subproof}
	\begin{claim}
		\label{claim:brooms2}
		$\chi(B\setminus Y)<q$.
	\end{claim}
	\begin{subproof}
		Let $\mac T$ be the family of all stable sets $S\subset A\setminus X$ with $\abs S=t$;
		and for each $S\in\mac T$, let $B_S$ be the set of vertices in $B\setminus Y$ with no neighbour in $S$.
		Each vertex in $B\setminus Y$ has at least $w^t$ nonneighbours in $A\setminus X$
		and so is anticomplete to some $S\in\mac T$ since $R(t,w)\le w^t$~\cite{MR1556929}.
		Hence $B\setminus Y= \bigcup_{S\in\mac T}B_S$.
		For each $S\in\mac T$,
		if $\chi(B_S)\ge w^{-t(t+2)}q$ then $S$ and $B_S$ satisfy the second outcome of the lemma, a contradiction;
		and so $\chi(B_S)< w^{-t(t+2)}q$.
		Hence
		\[\chi(B\setminus Y)< \abs{A\setminus X}^t\cdot w^{-t(t+2)}q
		\le(nw)^tw^{-t(t+2)}q
		=q.\qedhere\]
	\end{subproof}
	\cref{claim:brooms1,claim:brooms2} together verify the first outcome of the lemma.
	This completes the proof of \cref{lem:star}.
\end{proof}
We will also need the following simple extension of the fact that every graph $G$ has degeneracy at least $\chi(G)-1$.
\begin{lemma}
	\label{lem:chidelta}
	For every integer $p\ge1$, every graph $G$ with $\chi(G)>p$ has an induced subgraph $F$ with minimum degree at least $p$ and $\chi(F)\ge \chi(G)-p$.
\end{lemma}
\begin{proof}
	Let $\ell\ge0$ be maximal for which there are $v_1,\ldots,v_\ell\in V(G)$ such that for every $i\in[\ell]$,
	$v_i$ has fewer than $p$ neighbours in $V(G)\setminus\{v_1,\ldots,v_i\}$.
	Then $G[\{v_1,\ldots,v_\ell\}]$ has degeneracy less than $p$ and so $\chi(\{v_1,\ldots,v_\ell\})\le p<\chi(G)$.
	Let $F:=G\setminus\{v_1,\ldots,v_\ell\}$;
	then $\chi(F)\ge\chi(G)-p>0$ and $F$ has minimum degree at least $p$ by the maximality of $\ell$.
	This proves \cref{lem:chidelta}.
\end{proof}

We are now ready to prove \cref{thm:anti}, which we restate here for convenience.
\begin{theorem}
	\label{thm:brooms}
	For every $k,t\ge1$, there exists $d\ge1$ such that every non-complete graph $G$ with clique number $w\ge2$ contains either:
	\begin{itemize}
		\item an anticomplete pair $(A,B)$ with $\chi(A),\chi(B)\ge w^{-d}\chi(G)$; or
		
		\item an anticomplete pair $(P,Q)$ such that $G[P]$ is a $(k,t)$-broom and $\chi(Q)\ge w^{-d}\chi(G)$.
	\end{itemize}
\end{theorem}
\begin{proof}
	The proof first uses \cref{lem:covering} to obtain a long covering blockade, then iterates \cref{lem:star} inside the final block of the blockade to generate a long sequence of disjoint vertex subsets that together possess much chromatic number of the host graph while having a small sum of clique numbers (similar to the proof of \cref{thm:quasi}).
	
	We claim that $d:=6k+t(t+2)+9$ suffices.
	To this end,
	assume that the first outcome does not hold; then $\chi(G)\ge w^{d}$.
	By \cref{lem:covering}, $G$ contains either:
	\begin{itemize}
		\item an anticomplete pair $(A,B)$ with $\chi(A),\chi(B)\ge w^{-8k}\chi(G)$; or
		
		\item a $k$-covering blockade $(D_1,\ldots,D_k,E)$ with $\chi(D_k),\chi(E)\ge w^{-6k}\chi(G)$.
	\end{itemize}
	
	The first bullet cannot hold since the first outcome of the lemma fails; and so the second bullet holds.
	Let $\ell\ge0$ be maximal such that there are nonempty disjoint $E_0,E_1,\ldots,E_\ell\subset E$ satisfying:
	\begin{itemize}
		\item $\omega(E_0)+\omega(E_1)+\cdots+\omega(E_\ell)\le w$; and
		
		\item $\chi(E_0)+\chi(E_1)+\cdots+\chi(E_\ell)\ge \chi(E)-\ell(w^{-2}\chi(E)+3w^{t+2})$.
	\end{itemize}
	
	These are satisfied for $\ell=0$, taking $E_0=E$.
	For each $i\in\{0,1,\ldots,\ell\}$, since $E_i$ is nonempty, $\omega(E_i)\ge1$; and so $\ell< w$.
	Since $\chi(E)\ge w^{-6k}\chi(G)\ge w^{d-6k}\ge 3w^{t+4}$ by the choice of $d$, we see that
	\[\chi(E_1)+\cdots+\chi(E_{\ell})
	\ge \chi(E)-2\ell w^{-2}\chi(E)
	\ge \chi(E)-(1-w^{-1})\chi(E)
	= w^{-1}\chi(E).\]
	Thus, there exists $i\in\{0,1,\ldots,\ell\}$ with
	$$\chi(E_i)\ge \ell^{-1}(\chi(E_0)+\chi(E_1)+\cdots+\chi(E_{\ell}))
	\ge \ell^{-1}w^{-1}\chi(E).$$
	We may assume $i=0$.
	By \cref{lem:control} with $q=w^2$, $G[E_0]$ has a $w^2$-controlled induced subgraph $J$ with 
	\[\begin{aligned}
		\chi(J)\ge(1-w^{-3})\chi(E_0)
		&\ge \ell w^{-1}\chi(E_0)
		\ge w^{-2}\chi(E)\ge w^{-6k-2}\chi(G)\ge w^{d-6k-2}\ge2w^{t+7}
	\end{aligned}\]
	where the last inequality holds by the choice of $d$.
	Thus, \cref{lem:chidelta} gives an induced subgraph $F$ of $J$ with minimum degree at least $2w^{t+2}$ and 
	$$\chi(F)\ge\chi(J)-2w^{t+2}\ge w^{-2}\chi(E)-2w^{t+2}.$$
	
	The following property of $F$ is a consequence of the $w^2$-controlled property of $J$.
	\begin{claim}
		\label{claim:broom1} $\chi(F\setminus N_F[u])\ge w^{-5}\chi(J)$ for all $u\in V(F)$.
	\end{claim}
	\begin{subproof}
		Since $\chi(J)\ge 2w^{t+7}$,
		we see that $\chi(F)\ge \chi(J)-2w^{t+2}\ge (1-w^{-5})\chi(J)$.
		Hence,
		since $J$ is $w^{2}$-controlled, $\chi(N_J(u))<(1-w^{-4})\chi(J)$.
		Therefore, for every $u\in V(F)$,
		$$\begin{aligned}
			\chi(F\setminus N_F(u))&\ge\chi(F)-\chi(N_F(u)) \\
			&> (1-w^{-5})\chi(J)-(1-w^{-4})\chi(J)
			\ge w^{-5}\chi(J)>1=\chi(\{u\})
		\end{aligned}$$
		and so $\chi(F\setminus N_F[u])\ge w^{-5}\chi(J)$.
		This proves \cref{claim:broom1}.
	\end{subproof}
	Now, since $\chi(E)\ge w^{-6k}\chi(G)\ge w^{d-6k}$
	and $2w^{t+2-6k-d}\le w^{-3}$ by the choice of $d$, we have
	$$\begin{aligned}
		\chi(F)\ge w^{-2}\chi(E)-2w^{t+2}
		&\ge w^{-2}\chi(E)-2w^{t+2-6k-d}\chi(E)\\
		&\ge w^{-3}\chi(E)
		\ge w^{-6k-3}\chi (G)
		\ge w^{-d}\chi(G).
	\end{aligned}$$
	Thus, by the definition of covering blockades, the set $Z$ of vertices $z\in D_k$ with $\chi(F\setminus N_G(z))<w^{-1}\chi(F)$ satisfies $\chi(Z)<(1-w^{-2})\chi(D_k)$.
	Then
	$\chi(D_k\setminus Z)>w^{-2}\chi(D_k)\ge w^{-6k-2}\chi(G)\ge w^{-d}\chi(G)$.
	Hence, since the first outcome of the theorem fails,
	there exists $v\in D_k\setminus Z$ with a neighbour $u\in V(F)$.
	Since $u$ has degree at least $2w^{t+2}$ in $F$,
	there exists $A\subset N_F(u)$ such that $\abs A\ge w^{t+2}$ and $v$ is pure to $A$.
	Let $B:=V(F)\setminus N_G(v)$ if $v$ is complete to $A$, and let $B:=V(F)\setminus N_F[u]$ if $v$ is anticomplete to $A$;
	then $A,B$ are disjoint and $\chi(B)\ge \min(w^{-5}\chi(J),w^{-1}\chi(F))\ge w^{-5}\chi(J)$ by \cref{claim:broom1}.
	\begin{claim}
		\label{claim:broom2}
		$F$ contains an anticomplete pair $(P,Q)\subset (A,B)$ such that $P$ is a stable set of size $t$ and $\chi(Q)\ge w^{-d}\chi(G)$.
	\end{claim}
	\begin{subproof}
		Let $s:=w^{-6k-7}\chi(G)\le w^{-7}\chi(E)\le w^{-5}\chi(J)\le \chi(B)$.
		By \cref{lem:star}, $F$ contains either:
		\begin{itemize}
			\item a pair $(X,Y)\subset(A,B)$ with $\omega(X)+\omega(Y)\le \omega(F)$, $\abs{A\setminus X}<w^{t+2}$, and $\chi(B\setminus Y)<s$; or
			
			\item an anticomplete pair $(P,Q)\subset (A,B)$ where $P$ is a stable set of size $t$ and $\chi(Q)\ge w^{-t(t+2)}s$.
		\end{itemize}
		
		If the second bullet holds then we are done since the choice of $d$ yields
		$$\chi(Q)\ge w^{-t(t+2)}s=w^{-t(t+2)-6k-7}\chi(G)\ge w^{-d}\chi(G).$$
		Thus, suppose that the first bullet holds.
		Then since $\abs{A\setminus X}<w^{t+2}\le\abs A$ and $\chi(B\setminus Y)<q\le \chi(B)$,
		we see that $X,Y$ are nonempty.
		Because
		\[\chi(F\setminus (X\cup Y))
		\le \chi(A\setminus X)+\chi(B\setminus Y)
		<w^{t+2}+s
		\le w^{t+2}+w^{-7}\chi(E)\]
		we deduce that
		\begin{align*}
			\chi(X)+\chi(Y)
			&\ge \chi(F)-w^{t+2}-w^{-7}\chi(E)\\
			&\ge \chi(J)-3w^{t+2}-w^{-7}\chi(E)\\
			&\ge (1-w^{-3})\chi(E_0)-3w^{t+2}-w^{-7}\chi(E)\\
			&\ge \chi(E_0)-2w^{-3}\chi(E)-3w^{t+2}
			\ge \chi(E_0)-w^{-2}\chi(E)-3w^{t+2}.
		\end{align*}
		It follows that
		$$\sum_{j=1}^\ell\omega(E_j)+\omega(X)+\omega(Y)\le
		\sum_{j=1}^\ell\omega(E_j)+\chi(F) \le\sum_{j=0}^\ell\omega(E_j)\le w,$$
		and 
		\begin{align*}
			\sum_{1\le j\le\ell}\chi(E_j)
			+\chi(X)+\chi(Y)
			&=\sum_{0\le j\le \ell}\chi(E_j)
			+(\chi(X)+\chi(Y)-\chi(E_0))\\
			&\ge \chi(E)-(\ell+1)(w^{-2}\chi(E)+3w^{t+2})
		\end{align*}
		and so $E_0,E_1,\ldots,E_{i-1},X,Y,E_{i+1},\ldots,E_\ell$ contradict the maximality of $\ell$.
		This completes the proof of \cref{claim:broom2}.
	\end{subproof}
	
	Now, if $v=v_k$ is complete to $P$, then $\{v_1,\ldots,v_k\}\cup P$ and $Q$ satisfy the second outcome of the theorem;
	and if $v=v_k$ is anticomplete to $P$, then $\{v_2,\ldots,v_k,u\}\cup P$ and $Q$ do.
	This proves~\cref{thm:brooms}.
\end{proof}

\section{Polynomial versus linear complete pairs in $P_5$-free graphs}
\label{sec:llp5}
\subsection{Basic facts}
Due to its relevance in this section, we will reproduce the well-known Gy\'arf\'as path argument~\cite{MR382051}, as follows.
\begin{theorem}
	[Gy\'arf\'as]
	\label{thm:gs}
	For every $k\ge4$, every $P_k$-free graph $G$ with $\chi(G)\ge 2$ has a vertex $v$ with $\chi(N_G(v))\ge \frac1{k-2}\chi(G)$.
	Consequently, for every $w\ge 2$, if $\omega(G)\le w$ then $\chi(G)\le (k-2)^{w-1}$.
\end{theorem}
\begin{proof}
	Suppose that the first assertion is not true.
	We may assume $G$ is connected.
	Let $v\in V(G)$;
	then $\chi(G\setminus N_G(v))>\frac{k-3}{k-2}\chi(G)\ge\frac12\chi(G)\ge1$.
	Thus $G\setminus N_G[v]$ has a component with chromatic number $\chi(G\setminus N_G(v))>\frac{k-3}{k-2}\chi(G)$.
	Since $G$ is connected,
	every component of $G\setminus N_G[v]$ has a vertex with a neighbour in $N_G(v)$.
	Thus, there exists $\ell\in\{2,\ldots,k-2\}$ maximal for which there is an induced path $v_1\text-v_2\text-\cdots\text-v_\ell$ in $G$ and a connected induced subgraph $D$ of $G\setminus\{v_1,\ldots,v_\ell\}$ such that:  $\chi(D)>\frac{k-1-\ell}{k-2}\chi(G)$,
	$\{v_1,\ldots,v_{\ell-1}\}$ is anticomplete to $V(D)$, and $v_\ell$ has a neighbour in $V{(D)}$.
	Because
	$$\chi(D\setminus N_G(v_\ell))\ge \chi(D)-\chi(N_G(v_\ell))>\frac{k-2-\ell}{k-2}\chi(G),$$
	there is a component $D'$ of $D\setminus N_G(v_\ell)$ with $\chi(D')=\chi(D\setminus N_G(v_\ell))>\frac{k-2-\ell}{k-2}\chi(G)$.
	Since $D$ is connected, there exists $v_{\ell+2}\in V(D')$ with a neighbour $v_{\ell+1}\in N_G(v_\ell)\cap V(D)$.
	Then $v_1\text-v_2\text-\cdots\text-v_\ell\text-v_{\ell+1}\text-v_{\ell+2}$ is an induced path in $G$; and so $\ell<k-2$ since $G$ is $P_k$-free.
	But then $v_1\text-v_2\text-\cdots\text-v_\ell\text-v_{\ell+1}$ and $D'$ contradict the maximality of $\ell$.
	
	That proves the first assertion of the theorem; and the second one follows by induction on $w$, noting that the neighbourhood of $v$ has clique number at most $w-1$. 
	The proof of \cref{thm:gs} is complete.
\end{proof}
	For a graph $G$, a vertex $v\in V(G)$ is {\em mixed} on $S\subset V(G)\setminus\{v\}$ if it has a neighbour and a nonneighbour in $G$.
	The following simple fact about $P_5$-free graphs will be used frequently in the rest of the paper.
\begin{lemma}
	\label{lem:mixed}
	For every $P_5$-free graph $G$ and every anticomplete pair $(A,B)$ in $G$ with $A,B$ nonempty, no vertex $v\in V(G)$ is mixed on both $A$ and $B$.
\end{lemma}
\begin{proof}
	Suppose not.
	Then there are $a_1a_2\in E(G[A])$ and $b_1b_2\in E(G[B])$ with $a_1v,b_1v\in E(G)$ and $a_2v,b_2v\nin E(G)$;
	and so $a_2\text-a_1\text-v\text-b_1\text-b_2$ would be an induced $P_5$ in $G$, contrary to the $P_5$-freeness of $G$.
	This proves \cref{lem:mixed}.
\end{proof}
\subsection{Colourful induced subgraphs}
In what follows, for $\eps>0$, say that a graph $G$ is {\em $\eps$-colourful} if $\chi(G\setminus N_G[v])<\eps\cdot\chi(G)$ for all $v\in V(G)$.
The proof method of \cref{thm:mainll} is via the following result.
\begin{lemma}
	\label{lem:locdense}
	There exists $a\ge6$ such that for every $\eps\in(0,\frac12)$, every $P_5$-free graph $G$ with clique number at most $w\ge2$ contains either:
	\begin{itemize}
		\item an $\eps$-colourful induced subgraph $J$ with $\chi(J)\ge 2^{-6}\chi(G)$; or
		
		\item a complete pair $(A,B)$ with $\chi(A)\ge w^{-a}\chi(G)$ and $\chi(B)\ge 2^{-8}\eps\cdot \chi(G)$.
	\end{itemize}
\end{lemma}
We actually conjecture that the second outcome of this lemma can be dropped (with $2^{-6}$ in the first outcome replaced by some constant depending on $\eps$ only); and more generally the improved statement remains true with $P_5$ replaced by any forest, as follows:
\begin{conjecture}
	\label{conj:clful}
	For every $\eps>0$ and every forest $T$, there exists $\delta>0$ such that every $T$-free graph $G$ has an $\eps$-colourful induced subgraph $F$ with $\chi(F)\ge\delta\cdot\chi(G)$.
\end{conjecture}
In other words, this conjecture says that every graph with no copy of a given forest contains a \dd locally dense\ee{} induced subgraph with linear chromatic number.
If true, \cref{conj:clful} would be an analogue for chromatic number of R\"odl's theorem~\cite{MR837962} that every graph with a forbidden induced subgraph contains a linear-sized induced subgraph with very high minimum degree or very low maximum degree; but we have not been able to decide it when $T=P_5$ or even when $T$ is the two-edge matching.
It is not hard to see that \cref{conj:clful} holds for $T=P_4$, and a simple argument proves it (with $\delta=\abs \eps^{\abs T}/\abs T$) when $T$ is a star (we omit the proof).
When $T=P_5$, it would already be quite interesting if the following is true:
\begin{conjecture}
	\label{conj:modp5}
	There exists $\delta>0$ such that every $P_5$-free graph $G$ has an induced subgraph $F$ such that $\chi(F)\ge\delta\cdot\chi(G)$ and $\chi(N_F(v))\ge\delta\cdot \chi(F)$ for all $v\in V(F)$.
\end{conjecture}

Back to \cref{lem:locdense}: let us now see how it implies \cref{thm:mainll} via the following.
\begin{lemma}
	\label{lem:densepair}
	Let $\eps\in(0,1)$, and let $G$ be an $\eps$-colourful $P_5$-free graph with $\chi(G)\ge2$ and clique number at most $w\ge2$.
	Then there is a complete pair $(A,B)$ in $G$ with $\chi(A)\ge w^{-32}\chi(G)$ and $\chi(B)\ge \frac{1-\eps}{2}\chi(G)$.
\end{lemma}
\begin{proof}
	Since $\chi(G\setminus v)\ge\chi(G)-1\ge\frac12\chi(G)$ for all $v\in V(G)$,
	we may assume that $G$ is not complete.
	Then \cref{lem:covering} (with $k=4$) gives an anticomplete pair $(A,B)$ in $G$ with $\chi(A),\chi(B)\ge w^{-32}\chi(G)$;
	and we may assume $G[A],G[B]$ are connected.
	Among all such pairs $(A,B)$ in $G$, choose $(A,B)$ with $\chi(A)+\chi(B)$ maximal; and subject to these, with $\abs A+\abs B$ maximal.
	Since $\eps<1$ and $G$ is $\eps$-colourful, $G$ is connected;
	and so there is a minimal nonempty cutset $S$ separating $A,B$ in $G$.
	By the maximality of $(A,B)$, $G[A],G[B]$ are components of $G\setminus S$ and $\chi(G\setminus S)=\max(\chi(A),\chi(B))$.
	Because $G$ is $\eps$-colourful, we have $\chi(A),\chi(B)\le \eps\cdot \chi(G)$;
	and so $\chi(S)\ge(1-\eps)\chi(G)$.
	Now, since $G$ is $P_5$-free, \cref{lem:mixed} and the minimality of $S$ together give a partition $(P,Q)$ of $S$ such that $P$ is complete to $A$ and $Q$ is complete to $B$.
	We may assume $\chi(P)\ge\chi(Q)$;
	then $\chi(P)\ge\frac12\chi(S)\ge\frac{1-\eps}{2}\chi(G)$ and we are done.
	This completes the proof of \cref{lem:densepair}.
\end{proof}
We can now finish the proof of \cref{thm:mainll}.
\begin{proof}
	[Proof of \cref{thm:mainll}, assuming \cref{lem:locdense}]
	Let $a$ be given by \cref{lem:locdense};
	we claim that $b:=\max(a,40)$ suffices.
	To see this, we may assume $\chi(G)\ge w^b$, for otherwise the theorem is true by the Gy\'arf\'as path theorem \ref{thm:gs}.
	By \cref{lem:locdense} with $\eps=\frac12$, either:
	\begin{itemize}
		\item $G$ has an $\frac12$-colourful induced subgraph $J$ with $\chi(J)\ge 2^{-6}\chi(G)$; or
		
		\item there is a complete pair $(A,B)$ in $G$ with $\chi(A)\ge w^{-a}\chi(G)$ and $\chi(B)\ge 2^{-9}\chi(G)$.
	\end{itemize}
	If the first bullet holds, then since $\chi(J)\ge 2^{-6}\chi(G)\ge 2$,
	\cref{lem:densepair} gives a complete pair $(A,B)$ in $J$ with 
	$\chi(A)\ge w^{-32}\chi(J)\ge 2^{-6}w^{-32}\chi(G)\ge w^{-b}\chi(G)$ and 
	$\chi(B)\ge \frac14\chi(J)\ge 2^{-8}\chi(G)$
	by the choice of $b$ and we are done.
	If the second bullet holds
	then we are also done.
	This proves \cref{thm:mainll}.
\end{proof}

As such, the rest of this paper deals with the proof of \cref{lem:locdense}.

\subsection{Terminal partitions in controlled $P_5$-free graphs}
Recall that for $q\ge w\ge 2$, a graph $G$ with clique number at most $w$ is {\em $q$-controlled} if it is connected and $\chi(N_G(v))\le(1-q^{-2})\chi(G)$ for all $v\in V(G)$.
In the rest of this paper we will drop \dd$w$-\ee{} from \dd$w$-controlled\ee{} for notational convenience, and will be interested in controlled $P_5$-free graphs.
For convenience, let us restate the following consequence of \cref{lem:control} with $q=w$.
\begin{lemma}
	\label{lem:ctrl}
	For every $w\ge2$, every graph $G$ with clique number at most $w$ has a controlled induced subgraph $F$ with $\chi(F)>(1-w^{-1})\chi(G)$.
\end{lemma}

Much of the argument in the rest of the paper deals with the following kind of partitions. 
For $p\ge0$ and for a connected graph $G$ with clique number at most $w\ge2$, there exists $k\ge0$ maximal such that there is a partition $(A_1,\ldots,A_k,B,D)$ of $V(G)$ satisfying:
\begin{itemize}
	\item $D$ is anticomplete to $A_1\cup\cdots\cup A_k$;
	
	\item each vertex in $B$ has a neighbour in $A_1\cup\cdots\cup A_k$;
	
	\item for each $i\in[k]$, the set $B_i$ of vertices in $B$ with a neighbour in $A_i$ satisfies $1\le\chi(B_i)\le w^{-4}\chi(G)$;
	
	\item $G[A_1],\ldots,G[A_k]$ are the components of $G\setminus (B\cup D)$,
	each with chromatic number at least $p$; and
	
	\item $\chi(D)\ge (1-w^{-2})\chi(G)$, and each vertex in $B$ has a neighbour in each component $C$ of $G[D]$ with $\chi(C)\ge(1-w^{-2})\chi(G)$.
\end{itemize}
(These conditions are satisfied for $k=0$, taking $B$ empty and $D=V(G)$.)
Such a partition is called a {\em $p$-terminal} partition of $G$.
Here is a useful property of terminal partitions in controlled $P_5$-free graphs: the last part in each such partition \dd occupies\ee{} much of the chromatic number of the graphs in question.
\begin{lemma}
	\label{lem:ter}
	Let $p\ge0$, and let $G$ be a controlled $P_5$-free graph with clique number at most $w$,
	with a $p$-terminal partition $(A_1,\ldots,A_k,B,D)$.
	Then 
	$G[D]$ has a unique component with chromatic number at least $(1-w^{-2})\chi(G)$, and
	$\chi(D)\ge(1-w^{-3})\chi(G)$.
\end{lemma}
\begin{proof}
	Since $G$ is connected, we may assume $B$ is nonempty.
	Since $\chi(D)\ge(1-w^{-2})\chi(G)$,
	there is a component $C$ of $G[D]$ with $\chi(C)\ge (1-w^{-2})\chi(G)$.
	Because $G$ is controlled, every vertex in $B$ is then mixed on $V(C)$.
	Thus, if there is another component $C'$ of $G[D]$ with $\chi(C')\ge(1-w^{-2})\chi(D)$,
	then every vertex in $B$ would be mixed on both $V(C)$ and $V(C')$,
	contrary to \cref{lem:mixed}.
	This proves the first statement of the lemma.
	
	To prove the second statement,
		let $I\subset [k]$ be minimal with $\bigcup_{i\in I}B_i=B$.
		By the minimality of $I$,
		for each $i\in I$ there exists $y_i\in B_i$ with no neighbour in $\bigcup_{j\in I\setminus\{i\}}A_j$.
		Suppose that there are distinct $i,j\in I$ with $y_iy_j\nin E(G)$.
		Since each of $y_i,y_j$ has a neighbour in $C$,
		there is an induced path $P$ between $y_i,y_j$ and with interior inside $C$.
		Let $z_i\in A_i$ be a neighbour of $y_i$ and $z_j\in A_j$ be a neighbour of $y_j$; 
		then $z_i\text-P\text-z_j$ would be an induced path of length at least four in $G$, a contradiction.
		Hence $\{y_i:i\in I\}$ is a clique in $G$ and so $\abs I\le w$,
		which yields
		$$\chi(B)\le\sum_{i\in I}\chi(B_i)\le \abs I\cdot w^{-4}\chi(G)\le w^{-3}\chi(G).$$
	Now, \cref{lem:mixed} implies that
	every vertex in $B$ is pure to each of $A_1,\ldots,A_k$.
	Hence, for each $i\in[k]$, $B_i$ is complete to $A_i$;
	and so $\chi(A_i)< (1-w^{-2})\chi(G)\le\chi(C)$ since $B_i$ is nonempty and $G$ is controlled.
	Therefore $\chi(D)=\chi(D\cup(A_1\cup\cdots\cup A_k))$; and so
	$$\chi(D)\ge \chi(G)-\chi(B)\ge(1-w^{-3})\chi(G),$$
	which verifies the second statement of the lemma.
	This proves \cref{lem:ter}.
\end{proof}
As the following lemma illustrates, $p$-terminal partitions naturally appear in controlled $P_5$-free graphs with high-chromatic anticomplete pairs,
which exist (for suitable choices of~$p$) by \cref{lem:covering};
and terminal partitions are useful because they provide high-chromatic complete pairs.
\begin{lemma}
	\label{lem:term}
	Let $p\ge0$, and let $G$ be a controlled $P_5$-free graph with clique number at most $w$.
	Assume that every induced subgraph $F$ of $G$ with $\chi(F)\ge(1-w^{-3})\chi(G)$ contains an anticomplete pair $(P,Q)$ in $F$ with $\chi(P),\chi(Q)\ge p$.
	Then $G$ contains a complete pair $(A,B)$ with $\chi(A)\ge w^{-4}\chi(G)$ and $\chi(B)\ge p$.
\end{lemma}
\begin{proof}
	Suppose not.
	Let $(A_1,\ldots,A_k,B,D)$ be a $p$-terminal partition of $G$.
	By \cref{lem:ter},
	$\chi(D)\ge (1-w^{-3})\chi(G)$
	and there is a unique component $C$ of $G[D]$ with $\chi(C)\ge(1-w^{-2})\chi(G)$;
	then $\chi(C)=\chi(D)\ge(1-w^{-3})\chi(G)$.
	By the hypothesis, there is an anticomplete pair $(P,Q)$ in $C$ with $\chi(P),\chi(Q)\ge p$;
	and we may assume $G[P],G[Q]$ are connected.
	Among all such anticomplete pairs $(P,Q)$ in $C$,
	choose $(P,Q)$ with $\chi(P)+\chi(Q)$ maximal;
	and subject to these, with $\abs P+\abs Q$ maximal.
	We may assume that $\chi(P)\ge\chi(Q)$.
	Since $C$ is connected, there exists a minimal nonempty cutset $S$ separating $P,Q$ in $C$;
	then every vertex in $S$ has a neighbour in each of $P,Q$.
	The following claim shows that $P$ \dd occupies\ee{} much of the chromatic number of $G$.
	\begin{claim}
		\label{claim:pair1}
		Every vertex in $S$ is mixed on $P$, and $G[P]$ is the unique component of $C\setminus S$ with chromatic number at least $(1-w^{-2})\chi(G)$.
	\end{claim}
	\begin{subproof}
		By the maximality of $(P,Q)$,
		$G[P],G[Q]$ are two of the components of $C\setminus S$
		and $\chi(C\setminus(S\cup P\cup Q))\le\chi(P)$.
		By \cref{lem:mixed},
		each vertex in $S$ is complete to at least one of $P,Q$.
		Hence, since $\chi(P),\chi(Q)\ge p$
		and by our supposition,
		$\chi(S)\le 2w^{-4}\chi(G)\le w^{-3}\chi(G)$.
		It follows that
		$$\chi(P)=\chi(C\setminus S)
		\ge\chi(C)-\chi(S)
		\ge (1-w^{-3})\chi(G)-w^{-3}\chi(G)
		\ge (1-w^{-2})\chi(G).$$
		Since $F$ is controlled, every vertex in $S$ is then mixed on $P$;
		and so $S$ is complete to $Q$ which yields $\chi(Q)<(1-w^{-2})\chi(G)$.
		Hence $G[P]$ is the unique component of $C\setminus S$ with chromatic number at least $(1-w^{-2})\chi(G)$.
		This proves \cref{claim:pair1}.
	\end{subproof}
	We now use the $P_5$-free hypothesis to \dd extend\ee{} $(A_1,\ldots,A_k,B,D)$, as follows.
	\begin{claim}
		\label{claim:pair2}
		Every vertex in $B$ has a neighbour in $P$.
	\end{claim}
	\begin{subproof}
		Suppose there exists $u\in B$ with no neighbour in $P$.
		Let $v$ be a neighbour of $u$ in $A_1\cup\cdots\cup A_k$.
		Since $u$ has a neighbour in $D$ and $G[D]$ is connected,
		$G$ has an induced path $R$ of length at least two from $u$ to $P$ such that $V(R)\setminus(P\cup\{u\})\subset D\setminus P$.
		If $R$ has length at least three then $v\text-R$ would be an induced path of length at least four in $G$, a contradiction.
		Thus $R$ has length two; and so $v$ has a neighbour $z\in S$.
		Since $z$ is mixed on $P$ by \cref{claim:pair1},
		there exists $xy\in E(G[P])$ with $zx\in E(G)$ and $zy\nin E(G)$;
		but then $v\text-u\text-z\text-x\text-y$ would be an induced $P_5$ in $G$, a contradiction.
		This proves \cref{claim:pair2}.
	\end{subproof}
	Now, since $\chi(P)\ge(1-w^{-2})\chi(G)$ and $G$ is controlled, \cref{claim:pair2} implies that every vertex in $B$ is mixed on $P$.
	Thus $B$ is pure to $Q$ by \cref{lem:mixed};
	and so the set $Z$ of vertices in $B$ with a neighbour in $Q$ is complete to $Q$.
	Since $\chi(Q)\ge p$,
	our supposition implies that $\chi(S\cup Z)\le w^{-4}\chi(G)$.
	Therefore $(A_1,\ldots,A_k,Q,B\cup S,D\setminus(Q\cup S))$ contradicts the maximality of $k$.
	This proves \cref{lem:term}.
\end{proof}
We remark that combining \cref{lem:covering,lem:ctrl,lem:term} gives a relaxation of \cref{thm:mainll}:
every $P_5$-free graph $G$ with clique number $w$ contains a complete pair $(A,B)$ with $\chi(A),\chi(B)\ge w^{-d}\chi(G)$ for some universal $d>0$,
which is enough to imply $\chi(G)\le w^{O(\log w)}$.
To the best of our knowledge, the proof of the bound $\chi(G)\le w^{\log w}$ by Scott, Seymour, and Spirkl~\cite{MR4648583} relies crucially on induction on $w$ and does not immediately give such a pair.

Given the setup of terminal partitions and in particular \cref{lem:term},
we can view \cref{lem:locdense} as a corollary of the following lemma, which is essentially \cref{lem:locdense} itself plus a linear-chromatic anticomplete pair outcome.

\begin{lemma}
	\label{lem:linanti}
	There exists $b\ge6$ such that for every $\eps\in(0,\frac12)$, every $P_5$-free graph $G$ with clique number at most $w\ge2$ contains either:
	\begin{itemize}
		\item an $\eps$-colourful induced subgraph $J$ with $\chi(J)\ge2^{-4}\chi(G)$;
		
		\item an anticomplete pair $(P,Q)$ with $\chi(P)\ge 2^{-4}\chi(G)$ and $\chi(Q)\ge2^{-4}\eps\cdot \chi(G)$; or
		
		\item a complete pair $(A,B)$ in $G$ with $\chi(A)\ge w^{-b}\chi(G)$ and $\chi(B)\ge2^{-8}\eps\cdot \chi(G)$.
	\end{itemize}
\end{lemma}

\begin{proof}
	[Proof of \cref{lem:locdense}, assuming \cref{lem:linanti}]
	Let $b\ge6$ be given by \cref{lem:linanti};
	we claim that $a:=b+2$ suffices.
	To see this, suppose that none of the outcomes of the lemma holds.
	By \cref{lem:ctrl}, $G$ has a controlled induced subgraph $F$ with $\chi(F)>(1-w^{-1})\chi(G)\ge\frac12\chi(G)$.
	We now show that every sufficiently high-chromatic induced subgraph of $F$ contains a linear-chromatic anticomplete pair, as follows.
	\begin{clm}
		{\ref*{lem:locdense}.1}
		\label{claim:loc}
		Every induced subgraph $J$ of $F$ with $\chi(J)\ge(1-w^{-3})\chi(F)$
		contains an anticomplete pair $(P,Q)$ with $\chi(P),\chi(Q)\ge 2^{-4}\eps\cdot\chi(J)
		\ge 2^{-5}\eps\cdot\chi(F)$.
	\end{clm}
	\begin{subproof}
		By \cref{lem:linanti} with $J$ in place of $G$, $J$ contains either:
		\begin{itemize}
			\item an $\eps$-colourful induced subgraph $L$ with $\chi(L)\ge 2^{-4}\chi(J)$;
			
			\item an anticomplete pair $(P,Q)$ with $\chi(P),\chi(Q)\ge2^{-4}\eps\cdot\chi(J)$; or
			
			\item a complete pair $(X,Y)$ with $\chi(X)\ge w^{-b}\chi(J)$ and $\chi(Y)\ge 2^{-6}\eps\cdot\chi(J)$.
		\end{itemize}
		
		If the first bullet holds then the first outcome of the lemma holds since $2^{-4}\chi(J)\ge 2^{-5}\chi(F)\ge 2^{-6}\chi(G)$;
		and if the third bullet holds then the third outcome of the lemma holds since $w^{-b}\chi(J)\ge\frac14w^{-b}\chi(G)\ge w^{-a}\chi(G)$
		and $2^{-6}\eps\cdot\chi(J)\ge2^{-8}\eps\cdot \chi(G)$.
		So the second bullet holds by our supposition.
		This proves \cref{claim:loc}.
	\end{subproof}
	Now, by \cref{claim:loc} and \cref{lem:term} with $p=2^{-5}\eps\cdot\chi(F)$,
	there is a complete pair $(A,B)$ in $F$ with $\chi(A)\ge w^{-4}\chi(F)\ge \frac12w^{-4}\chi(G)\ge w^{-a}\chi(G)$ and 
	$\chi(B)\ge p=2^{-5}\eps\cdot\chi(F)\ge 2^{-6}\eps\cdot\chi(G)$,
	which verifies the second outcome of the lemma, a contradiction.
	This proves \cref{lem:locdense}.
\end{proof}

\subsection{Colourful induced subgraphs versus high-$\chi$ pure pairs}
\label{sec:locdense}

The purpose if this section is to prove \cref{lem:linanti} and in turn finish the proof of \cref{thm:mainll}.
We require the following application of the \erh{} property of $P_5$ (see \cref{thm:ehp5}).
\begin{lemma}
	\label{lem:mid}
	There exists $d\ge6$ such that the following holds.
	Let $G$ be a $P_5$-free graph with $\omega(G)=w\ge2$,
	and let $P,Q\subset V(G)$ be nonempty,
	such that $\chi(Q\setminus N_G(u))\le w^{-d}\chi(Q)$ for all $u\in P$.
	Then the set of vertices in $Q$ with a nonneighbour in $P$ has chromatic number at most $w^{-2}\chi(Q)$.
\end{lemma}
\begin{proof}
	Let $a\ge 4$ be given by \cref{thm:ehp5};
	we claim that $d:=a+2$ satisfies the lemma.
	To see this, let $T$ be the set of vertices in $Q$ with a nonneighbour in $P$.
	For each $u\in P$, let $T_u:=Q\setminus N_G(u)$;
	then $\chi(T_u)<w^{-d}\chi(Q)$.
	Since $T=\bigcup_{u\in P}T_u$,
	there exists $S\subset P$ minimal such that $T=\bigcup_{u\in S}T_u$.
	Let $S=\{u_1,\ldots,u_t\}$.
	By the minimality of $S$, for every $i\in[t]$ there exists $z_i\in T_{u_i}$ such that $z_i$ is nonadjacent to $u_i$ and complete to $S\setminus\{u_i\}$.
	Let $I\subset[t]$ be such that $\{u_i:i\in I\}$ is a largest stable set in $G[S]$.
	Suppose that $\abs I>w\ge 2$.
	If there are distinct $i,j\in I$ with $z_iz_j\nin E(G)$,
	then for some $\ell\in I\setminus\{i,j\}$, $u_i\text-z_i\text-u_\ell\text-z_j\text-u_j$ would be an induced $P_5$ in $G$, a contradiction.
	Thus $\alpha(S)=\abs I\le w$;
	and so the choice of $a$ implies $\abs S\le w^a$.
	Hence $\chi(T)\le \abs S\cdot w^{-d}\chi(Q)\le w^{a-d}\chi(Q)=w^{-2}\chi(Q)$.
	This proves \cref{lem:mid}.
\end{proof}

The proof of \cref{lem:linanti} first picks a vertex whose neighbourhood has linear chromatic number and whose nonneighbourhood has polynomial chromatic number,
then extensively exploits the $P_5$-free hypothesis to examine the interactions between these two sides via a terminal partition of a controlled induced subgraph of the nonneighbourhood.
The colourful induced subgraph of linear chromatic number will then come out as a linear portion of the neighbourhood.
Let us go into detail.
\begin{proof}
	[Proof of \cref{lem:linanti}]
	Let $d\ge6$ be given by \cref{lem:mid};
	we claim that $b:=d+6$ suffices.
	Thus, suppose that no outcome of the lemma holds.
	By \cref{thm:gs},
	the set $Z$ of vertices $z$ in $G$ with $\chi(N_G(z))<2^{-3}\chi(G)$
	satisfies $\chi(Z)<\frac12\chi(G)$;
	and so $\chi(G\setminus Z)>\frac12\chi(G)$.
	By \cref{lem:ctrl},
	$G\setminus Z$ has a controlled induced subgraph $F$ with $\chi(F)>(1-w^{-1})\chi(G\setminus Z)>\frac14\chi(G)$.
	Thus, since the first outcome of the lemma does not hold,
	there exists $v\in V(F)$ such that the set $Q$ of nonneighbours of $v$ in $G$ satisfies $\chi(Q)\ge\eps\cdot\chi(F)$;
	and since $F$ is controlled,
	$\textstyle\chi(Q)\ge\max(w^{-2}\chi (F),\eps\cdot \chi(F))
	\ge\frac14\max(w^{-2},\eps)\chi(G)$.
	By \cref{lem:ctrl}, there exists $S\subset Q$ such that $G[S]$ is controlled and 
	$$\chi(S)> (1-w^{-1})\chi(Q)\ge2^{-3}\max(w^{-2},\eps)\chi(G).$$
	
	Let $(A_1,\ldots,A_k,B,D_0)$ be a $(\frac12w^{-d}\chi(G))$-terminal partition of $G[S]$.
	By \cref{lem:ter}, $\chi(D_0)\ge(1-w^{-3})\chi(S)$
	and there is a unique component $G[D]$ of $G[D_0]$ with $\chi(D)\ge(1-w^{-2})\chi(S)$;
	then $\chi(D)\ge(1-w^{-3})\chi(S)$.
	Let $X$ be the set of vertices in $N_G(v)$ with no neighbour in $D$,
	let $Y$ be the set of vertices $u\in N_G(v)$ with $\chi(D\setminus N_G(u))<w^{-d}\chi(D)$,
	and let $R:=N_G(v)\setminus(X\cup Y)$.
	Note that 
	$$\chi(D)\ge(1-w^{-3})\chi(S)>2^{-4}\max(w^{-2},\eps)\chi(G).$$
	Hence, since the second outcome of the lemma does not hold,
	$\chi(X)\le 2^{-4}\chi(G)\le\frac14\chi(N_G(v))$.
	\begin{clm}
		{\ref*{lem:linanti}.1}
		\label{claim:linanti1} $\chi(R)\ge\frac12\chi(N_G(v))\ge 2^{-4}\chi(G)$.
	\end{clm}
	\begin{subproof}
		By \cref{lem:mid} and the choice of $b$,
		the set of vertices in $D$ complete to $Y$
		has chromatic number at least
		$(1-w^{-2})\chi(D)\ge (1-w^{-2})^2\chi(S)\ge \frac12\chi(S)\ge
		2^{-5}\eps\cdot\chi(G)$.
		Hence $\chi(Y)\le w^{-b}\chi(G)\le 2^{-6}\chi(G)\le2^{-2}\chi(N_G(v))$ since the third outcome of the lemma does not hold.
		Then
		$\chi(R)\ge \chi(N_G(v))-\chi(X\cup Y)\ge\frac12\chi(N_G(v))\ge 2^{-4}\chi(G)$.
		This proves \cref{claim:linanti1}.
	\end{subproof}
	By \cref{claim:linanti1} and since the first outcome of the lemma does not hold,
	there exists $u\in R$ such that the set $E$ of nonneighbours of $u$ in $R$ satisfies $\chi(E)\ge\eps\cdot \chi(R)\ge 2^{-4}\eps\cdot\chi(G)$.
	Let $T:=N_G(u)\cap D$, and let
	$C$ be a component of $G[D\setminus N_G(u)]$ with
	$$\chi(C)=\chi(D\setminus N_G(u))\ge w^{-d}\chi(D).$$
	\begin{clm}
		{\ref*{lem:linanti}.2}
		\label{claim:mixedll1}
		$E\cup T$ is pure to $V(C)$.
	\end{clm}
	\begin{subproof}
		First, if there exists $z\in E$ mixed on $V(C)$,
		then there would be $xy\in E(C)$ such that $zx\in E(G)$ and $zy\nin E(G)$;
		and so $u\text-v\text-z\text-x\text-y$ would be an induced $P_5$ in $G$, a contradiction.
		Second, if there exists $z\in T$ mixed on $V(C)$,
		then there would be $xy\in E(C)$ with $zx\in E(G)$ and $zy\nin E(G)$;
		and so $v\text-u\text-z\text-x\text-y$ would be an induced $P_5$ in $G$, a contradiction.
		Hence $E\cup T$ is pure to $V(C)$.
		This proves \cref{claim:mixedll1}.
	\end{subproof}
	Let $E_1$ be the set of vertices in $E$ with no neighbour in $V(C)$.
	\begin{clm}
		{\ref*{lem:linanti}.3}
		\label{claim:mixedll}
		$\chi(E_1)\ge\frac12\chi(E)\ge\frac12\eps\cdot\chi(R)$.
	\end{clm}
	\begin{subproof}
		By \cref{claim:mixedll1}, $E\setminus E_1$ is complete to $V(C)$. Thus, since
		$$\textstyle\chi(C)\ge w^{-d}\chi(D)\ge 2^{-4}w^{-d-2}\chi(G)
		\ge w^{-d-6}\chi(G)
		\ge w^{-b}\chi(G)$$
		by the choice of $b$,
		and since the third outcome of the lemma does not hold,
		we have
		\[\chi(E_1)\ge\chi(E)-\chi(E\setminus E_1)\ge\chi(E)- 2^{-6}\eps\cdot\chi(G)\ge\chi(E)/2\ge\eps\cdot \chi(R)/2.\qedhere\]
	\end{subproof}
	Let $U$ be the set of vertices in $T$ with a neighbour in $V(C)$.
	By \cref{claim:mixedll1}, $U$ is complete to $V(C)$;
	and since $G[D]$ is connected, $U$ is nonempty.
	Let $W$ be the set of vertices in $B$ with a neighbour in $V(C)$.
	We aim to bound $\chi(U\cup W)$.
	To do this, let $W_1:=W\setminus N_G(u)$,
	and let $W_2:=W\cap N_G(u)$.
	Let $C_1$ be a component of $G[W_1]$ with $\chi(C_1)=\chi(W_1)$.
	The following claim and \cref{claim:mixedll} together show that $\chi(U\cup W_2)$ is small.
	\begin{clm}
		{\ref*{lem:linanti}.4}
		\label{claim:mixedll2}
		$E_1$ is complete to $U\cup W_2$ and pure to $V(C_1)$.
	\end{clm}
	\begin{subproof}
		First, if there are $x\in E_1$ and $y\in U\cup W_2$ with $xy\nin E(G)$,
		then $x\text-v\text-u\text-y\text-z$ would be an induced $P_5$ in $G$ for some $z\in V(C)$, a contradiction.
		Second, if there exists $x\in E_1$ mixed on $V(C_1)$, then there are $yz\in E(C_1)$ with $xy\in E(G)$ and $xz\nin E(G)$;
		and so $u\text-v\text-x\text-y\text-z$ would be an induced $P_5$ in $G$, a contradiction.
		Thus $E_1$ is complete to $U\cup W_2$ and pure to $V(C_1)$.
		This proves \cref{claim:mixedll2}.
	\end{subproof}
	Let $E_2$ be the set of vertices in $E_1$ with a nonneighbour in $V(C_1)$ (so $E_2$ is empty if $V(C_1)$ is empty).
	By \cref{claim:mixedll2}, $E_2$ is anticomplete to $V(C_1)$.
	\begin{clm}
		{\ref*{lem:linanti}.5}
		\label{claim:mixedll3}
		$\chi(E_2)\le 2^{-6}\eps\cdot\chi(G)$.
	\end{clm}
	\begin{subproof}
		If $V(C_1)$ is empty then this is true;
		so we may assume there exists $y\in V(C_1)$.
		By \cref{lem:ter}, there exists $i\in[k]$ such that $y$ is complete to $A_i$.
		
		Suppose that there are $x\in E_2$ and $z\in A_i$ with $xz\nin E(G)$.
		Let $r\in U$, and let $t\in V(C)$ be a neighbour of $y$.
		If $yr\nin E(G)$,
		then $v\text-x\text-r\text-t\text-y$ would be an induced $P_5$ in $G$;
		and if $yr\in E(G)$,
		then $v\text-x\text-r\text-y\text-z$ would be an induced $P_5$ in $G$, a contradiction.
		Thus, $E_2$ is complete to $A_i$.
		Since $\chi(A_i)\ge\frac12 w^{-d}\chi(S)\ge w^{-b}\chi(G)$ and the third outcome of the lemma does not hold,
		$\chi(E_2)\le 2^{-6}\eps\cdot\chi(G)$.
		This proves \cref{claim:mixedll3}.
	\end{subproof}
	We are now ready to bound $\chi(U\cup W)$.
	\begin{clm}
		{\ref*{lem:linanti}.6}
		\label{claim:mixedll4}
		$\chi(U)\le w^{-b}\chi(G)$ and $\chi(W)\le 2w^{-b}\chi(G)$.
	\end{clm}
	\begin{subproof}
		By \cref{claim:mixedll,claim:mixedll3} and the choice of $b$, we have
		$$\chi(E_2)\le 2^{-6}\eps\cdot\chi(G)\le 2^{-2}\eps\cdot\chi(R)\le\chi(E_1)/2$$
		and so
		$\chi(E_1\setminus E_2)\ge\frac12\chi(E_1)\ge 2^{-6}\eps\cdot\chi(G)$.
		Thus, since $E_1\setminus E_2$ is complete to $V(C_1)\cup U\cup W_2$ by definition and \cref{claim:mixedll2},
		and since the third outcome of the lemma fails,
		$\chi(U),\chi(C_1),\chi(W_2)\le w^{-b}\chi(G)$.
		Hence
		\[\chi(W)=\chi(W_1\cup W_2)\le \chi(W_1)+\chi(W_2)=\chi(C_1)+\chi(W_2)\le 2w^{-b}\chi(G).\qedhere\]
	\end{subproof}
	
	The rest of the proof is similar to the proof of \cref{lem:term}.
	Let $D':=D\setminus(V(C)\cup U)$.
	Because $G[S]$ is controlled and $U$ is nonempty, we have $\chi(C)<(1-w^{-2})\chi(S)$.
	Thus, since
	$\chi(U)\le w^{-b}\chi(G)\le 2^4w^{2-b}\chi(S)\le w^{-3}\chi(S)$
	by \cref{claim:mixedll4} and the choice of $b$, we obtain
	$$\chi(D'\cup V(C))=\chi(D\setminus U)
	\ge (1-w^{-3})\chi(S)-w^{-3}\chi(S)
	\ge(1-w^{-2})\chi(S)>\chi(C).$$
	Hence, since $D'$ is anticomplete to $V(C)$, we deduce that 
	$\chi(D')=\chi(D\setminus U)\ge (1-w^{-2})\chi(S)$;
	and so there is a component $C'$ of $D'$ with $\chi(C')\ge(1-w^{-2})\chi(S)$.
	Let $U'$ be the set of vertices in $U$ with a neighbour in $V(C')$;
	then $U'$ is nonempty since $G[D]$ is connected.
	Since $V(C)$ is complete to $U$, there exists a unique component $C_0$ of $G[D\setminus U']$ with $V(C)\subset V(C_0)$.
	We now show that much of the chromatic number of $S$ is \dd concentrated\ee{} on $C'$.
	\begin{clm}
		{\ref*{lem:linanti}.7}
		\label{claim:mixedll5}
		Every vertex in $U'$ is mixed on $C'$,
		and $C'$ is the unique component of $G[D\setminus U']$ with chromatic number at least $(1-w^{-2})\chi(S)$.
	\end{clm}
	\begin{subproof}
		Since $G[S]$ is controlled, every vertex in $U'$ is mixed on $V(C')$.
		Then the components of $D\setminus U'$ consist of $C_0$ and the components of $D'$ with no neighbour in $U\setminus U'$ (these include $C'$).
		By \cref{lem:mixed}, every vertex in $U'$ is pure to every component of $D\setminus U'$ different from $C'$.
		Hence every such component $K$ is complete to some vertex in $U'$,
		which yields $\chi(K)<(1-w^{-2})\chi(S)$ since $G[S]$ is controlled.
		It follows that $C'$ is the unique component of $G[D\setminus U']$ with chromatic number at least $(1-w^{-2})\chi(S)$.
		This proves \cref{claim:mixedll5}.
	\end{subproof}
	It suffices to \dd extend\ee{} the terminal partition $(A_1,\ldots,A_k,B,D_0)$ via the following.
	\begin{clm}
		{\ref*{lem:linanti}.8}
		\label{claim:mixedll6}
		Every vertex in $B$ has a neighbour in $V(C')$.
	\end{clm}
	\begin{subproof}
		Suppose that there exists $y\in B$ with no neighbour in $V(C')$.
		Let $z\in A_1\cup\cdots\cup A_k$ be a neighbour of $y$.
		Then since $y$ has a neighbour in $D$ and $G[D]$ is connected,
		$G$ has an induced path $P$ of length at least two from $y$ to $V(C')$ with $V(P)\setminus(V(C')\cup\{y\})\subset D\setminus V(C')$.
		If $P$ has length at least three then $z\text-P$ would be an induced path of length at least four in $G$, a contradiction.
		Thus $P$ has length two; and so $y$ has a neighbour $x\in U'$.
		By \cref{claim:mixedll5}, there exists $rt\in E(C')$ with $xr\in E(G)$ and $xt\nin E(G)$.
		But then $z\text-y\text-x\text-r\text-t$ would be an induced $P_5$ in $G$, a contradiction.
		This proves \cref{claim:mixedll6}.
	\end{subproof}
	
	Now, let $A_{k+1}:=V(C_0)$, and let
	$W'$ be the set of vertices in $B$ with a neighbour in $A_{k+1}$;
	then $W\subset W'$.
	Since every vertex in $W'$ is mixed on $V(C')$, \cref{lem:mixed} implies that
	$W'$ is complete to $A_{k+1}\supset V(C)$;
	and so $W'=W$.
	Let $B_{k+1}:=U'\cup W\subset U\cup W$;
	then
	$$\chi(B_{k+1})\le\chi(U\cup W)\le 3w^{-b}\chi(G)\le3\cdot 2^4w^{2-b}\chi(S) \le 2^6w^{2-b}\chi(S)\le w^{-4}\chi(S)$$
	by \cref{claim:mixedll4} and the choice of $b$.
	Hence, since $\chi(A_{k+1})\ge\chi(C)\ge w^{-d}\chi(D)\ge\frac12w^{-d}\chi(S)$,
	the partition $(A_1,\ldots,A_k,A_{k+1},B\cup U',D\setminus(U'\cup V(C_0)))$
	would violate the maximality of $k$, a contradiction.
	This proves \cref{lem:linanti}.
\end{proof}
Finally, we would like to remark that the arguments in this section can be adapted to deduce a \dd polynomial versus linear near-complete pairs\ee{} result for excluding $(4,t)$-brooms for every $t\ge1$:
there exists $b\ge1$ (depending on $t$) such that every $(4,t)$-broom-free graph $G$ with $\omega(G)=w$ contains disjoint $A,B\subset V(G)$ with $\chi(A)\ge w^{-b}\chi(G)$, $\chi(B)\ge 2^{-b}\chi(G)$, and $\chi(A\setminus N_G(v))<w^{-1}\chi(A)$ for all $v\in B$.
The reason why the above approach does not seem to yield full completeness in this case is because the proof of \cref{lem:mid} does not work for $(4,t)$-brooms when $t\ge2$ (even if this graph satisfies the \erh{} conjecture due to \cref{thm:ehp5} and a theorem of Alon, Pach, and Solymosi~\cite{MR1832443}),
and several arguments involving the mixed property would instead require the condition \dd having a neighbour and a nonneighbourhood with chromatic number at least $\chi(G)/\operatorname{poly}(w)$\ee.
Still, such a \dd near-complete\ee{} result would be enough to show that $(4,t)$-brooms satisfy similar bounds as in \cref{thm:llp5}.
We omit the detailed proofs, which are just technical adjustments of the presented material. 
\section*{Acknowledgements}
We would like to thank Alex Scott, Paul Seymour, and Raphael Steiner for helpful discussions.
A major portion of this paper appeared in the author's PhD thesis~\cite{2025thes}.

\bibliographystyle{abbrv}
\bibliography{2024hchi}
\end{document}